\documentclass[12pt]{amsart}
\usepackage{blindtext}
\usepackage{amsmath,graphicx}
\usepackage{adjustbox}
\usepackage{amssymb}
\usepackage{tikz, tikz-cd}
\usepackage{fullpage}
\usepackage[all]{xy}
\usepackage[margin=1in]{geometry}
\usepackage{amsthm}
\usepackage{amsfonts}
\usepackage{bbm}
\usepackage{eufrak}
\usepackage{comment}
\usepackage{amsmath}
\usepackage{amsthm}
\usepackage{bbm}
\usepackage{eufrak}
\usepackage{array} 
\usepackage{color}
\usepackage[utf8]{inputenc}
\usepackage[english]{babel}
\usepackage{hyperref}
\newcolumntype{L}{>{$}l<{$}}

\DeclareMathSymbol{\shortminus}{\mathbin}{AMSa}{"39}

\newtheorem{theorem}{Theorem}[section]
\newtheorem{lemma}[theorem]{Lemma}

\newtheorem{corollary}[theorem]{Corollary}

\newtheorem{proposition}[theorem]{Proposition}
\theoremstyle{definition}  
\newtheorem{definition} [theorem] {Definition}

\newtheorem{remark} [theorem] {Remark}
\newtheorem{question} [theorem] {Question}

\theoremstyle{definition}

\newcommand{\F}{{\mathbb{F}}}

\newcommand{\Q}{{\mathbb{Q}}}

\newcommand{\R}{{\mathbb{R}}}
\newcommand{\Z}{{\mathbb{Z}}}

\newcommand{\surj}{\twoheadrightarrow}
\newcommand{\inj}{\hookrightarrow}

\newcommand{\spin}{\text{spin}}

\DeclareMathOperator{\HFL}{HFL}
\DeclareMathOperator{\HFK}{HFK}

\DeclareMathOperator{\HF}{HF}

\DeclareMathOperator{\CFK}{CFK}
\DeclareMathOperator{\CF}{CF}

\DeclareMathOperator{\rank}{rank}

\usepackage[textsize=scriptsize]{todonotes}

\usepackage{enumitem}
\usepackage{bm}

\calclayout

\newif\ifrevision
\newcommand{\change}[1]{%
\ifrevision
\textcolor{blue}{#1}%
\else
#1%
\fi
}
\usepackage[pagewise]{lineno}

\title{The $\CFK^\infty$ type of Almost $L$-space knots}
\author{Fraser Binns}
\address{Department of Mathematics, Princeton University}
\email{fb1673@princeton.edu}
\date{\today}

\begin{document}
\revisionfalse     

\begin{abstract}
    Heegaard Floer homology and knot Floer homology are powerful invariants of $3$-manifolds and links respectively. $L$-space knots are knots which admit Dehn surgeries to $3$-manifolds with Heegaard Floer homology of minimal rank. In this paper we study almost $L$-space knots, which are knots admitting large Dehn surgeries to $3$-manifolds with Heegaard Floer homology of next-to-minimal rank. Our main result is a classification of the $\CFK^\infty(-)$ type of almost $L$-space knots. As corollaries we show that almost $L$-space knots satisfy various topological properties, including some given by Baldwin-Sivek~\cite{baldwin2022characterizing}. We also give some new cable link detection results.
\end{abstract}
\maketitle

\section{Introduction}

Heegaard Floer homology is a package of invariants in low dimensional topology due to Ozsv\'ath-Szab\'o~\cite{ozsvath2004holomorphic}. We will be interested in the $3$-manifold invariant, which is due to Ozsv\'ath-Szab\'o~\cite{ozsvath2004holomorphic}, and the knot invariant which is due independently to Ozsv\'ath-Szab\'o~\cite{Holomorphicdisksandknotinvariants} and J.Rasmussen~\cite{Rasmussen}. The three manifold invariant, denoted $\widehat{\HF}(-)$, satisfies the following rank inequality;

\begin{equation}\label{H1rankbound}
    \rank(\widehat{\HF}(Y))\geq |H_1(Y;\Z)|
\end{equation}

Here $Y$ is a rational homology sphere and $|H_1(Y;\Z)|$ is the number of elements in $H_1(Y;\Z)$. This equation results from the fact that an appropriate \change{decategorification} of $\widehat{\HF}(Y)$ yields $|H_1(Y;\Z)|$. An \emph{$L$-space} is a rational homology sphere for which inequality~\ref{H1rankbound} is tight. Understanding the geometric and algebraic properties of $L$-spaces is of central interest in low dimensional topology and the subject of Boyer-Gordon-Watson's ``$L$-space conjecture"~\cite{boyer2013spaces}.

\emph{Dehn-surgery} is a well studied operation on $3$-manifolds. Let $q\in \Q$ and $K$ be a knot in $S^3$. $q$-surgery on $K$, denoted $S^3_q(K)$, is the manifold obtained by removing a tubular neighborhood of $K$ from $S^3$, then regluing it with framing determined by $q$. An \emph{$L$-space knot} is a knot $K$ for which $S^3_n(K)$ is an $L$-space for for some $n\in\Z^{\geq 0}$. The knot Floer homology of $L$-space knots is well understood, by work of Ozsv\'ath-Szab\'o~\cite{ozsvath2005knot}.

In this paper we will study \emph{almost $L$-space knots}. The following definitions are due to Baldwin-Sivek~\cite{baldwin2022characterizing}.

\begin{definition}
Let $Y$ be a rational homology sphere. We call $Y$ an \emph{almost $L$-space} if $\rank(\widehat{\HF}(Y))=|H_1(Y;\Z)|+2$.
\end{definition}

We note that there is no rational homology sphere with $\rank(\widehat{\HF}(Y))=|H_1(Y;\Z)|+1$, as the \change{decategorification} from $\widehat{\HF}(Y)$ to $|H_1(Y;\Z)|$ respects parity. Almost $L$-spaces are thus the rational homology spheres for which  inequality~\ref{H1rankbound} is ``almost tight". There is a question of whether or not this is quite the ``correct" definition for almost $L$-spaces. See Section~\ref{remarkssection} for some discussion.

\begin{definition}{\cite{baldwin2022characterizing}}
A knot $K$ in $S^3$ is called an \emph{almost $L$-space knot} if there exists an $n\geq 2g(K)-1$ such that $\rank(\widehat{\HF}(S^3_n(K)))=n+2$.
\end{definition}

As we will see, there are a number of equivalent definitions of almost $L$-space knots. The reason to include the condition $n\geq 2g(K)-1$ is that if $K$ is an $L$-space knot then $\rank(\widehat{\HF}(S^3_{2g(K)-2}(K)))=2g(K)$.

We turn our attention now to knot Floer homology. The strongest version of knot Floer homology is the filtered chain homotopy type of a chain complex $\CFK^\infty(-)$. This is a finitely generated chain complex over $\F[U,U^{-1}]$ equipped with two gradings; the $U$ grading and the Alexander grading.  Our main result is a classification of $\CFK^\infty(-)$ for almost $L$-space knots;

\begin{theorem}

\label{infinityclassification}

Let $K$ be an almost $L$-space knot. Then $\CFK^\infty(K)$ has the filtered chain homotopy type of one of the following complexes;\begin{enumerate}[label=(\Alph*)]
    \item\label{staircasebox} A staircase complex direct sum a box complex.
    \item\label{almoststaircase} An almost staircase complex.
\end{enumerate}

\end{theorem}

The definitions of the complexes referred to in this theorem are deferred to Section~\ref{review}. They are illustrated, however, in Figures~\ref{staircasefig}, \ref{boxfig}, \ref{almoststaircase1fig} and \ref{almoststaircase2fig}. \change{We also note that there are two distinct types of almost staircase complex; one as in Definition~\ref{def:almoststaircase1} and Figure~\ref{almoststaircase1fig} which we call ``almost staircases of type one" and another as in Definition~\ref{def:almoststaircase2} and Figure~\ref{almoststaircase2fig} which we call ``almost staircases of type 2". Moreover, in some degenerate cases almost staircase complexes are filtered chain homotopy equivalent to staircase plus box complexes.}
\begin{figure}[h]\label{staircasefig}
\begin{tikzcd}
    y_1 & x_1\arrow[l]\arrow[dd]\\\\& y_2&& x_2\arrow[ll]\arrow[d]\\&&& y_3
\end{tikzcd}
\caption{A staircase complex\change{, as in part~\ref{staircasebox} of Theorem~\ref{infinityclassification}.} The horizontal direction indicates the $U$ grading and the vertical direction indicates the $A$ grading.}
\end{figure}

\begin{figure}[h]
\begin{tikzcd}
    x_1\arrow[d] & x_2\arrow[l]\arrow[d]\\ x_3&x_4\arrow[l]
\end{tikzcd}

\caption{A box complex\change{, as in part~\ref{staircasebox} of Theorem~\ref{infinityclassification}.} The horizontal direction indicates the $U$ grading and the vertical direction indicates the $A$ grading. Note the arrows are of length one.}\label{boxfig}
\end{figure}

\begin{figure}[h]
\begin{tikzcd}
    x_{-2} & y_{-1}\arrow[l]\arrow[dd]\arrow[dl]\\
    x_{1}\arrow[d]&&y_{1}\arrow[ll]\arrow[dl]\arrow[d]\\ z&x_{-1}\arrow[l]&x_{2}
\end{tikzcd}
\caption{An almost staircase complex of type $1$\change{, as in part~\ref{almoststaircase} of Theorem~\ref{infinityclassification}.}. The horizontal direction indicates the $U$ grading and the vertical direction indicates the $A$ grading.}\label{almoststaircase2fig}
\end{figure}

\begin{figure}[h]
\begin{tikzcd}
    x_{-2} & y_{-2}\arrow[l]\arrow[dd]\\ &&y_1\arrow[dd]\arrow[dl]&z\arrow[l]\arrow[d]\\
    &x_{-1}&& y_{-1}\arrow[ll]\arrow[dl]\\&& x_1 && y_2\arrow[ll]\arrow[d]\\&&&& x_2
\end{tikzcd}

\caption{An almost staircase complex of type $2$\change{, as in part~\ref{almoststaircase} of Theorem~\ref{infinityclassification}.}. The horizontal direction indicates the $U$ grading and the vertical direction indicates the $A$ grading.}\label{almoststaircase1fig}
\end{figure}

We note in contrast that an $L$-space knot has $\CFK^\infty$ given by a staircase complex, as proven by Ozsv\'ath-Szab\'o~\cite{ozsvath2005knot}.

It is straightforward to construct examples of almost $L$-space knots with $\CFK^\infty(-)$ \change{type of almost staircase complexes of type 1} by taking $(2,4g(K)-3)$-cables of $L$-space knots. Examples of knots with $\CFK^\infty$ \change{type of box plus staircase complexes} include the figure eight, the mirror of $5_2$, $T(2,3)\# T(2,3)$, $10_{139}$, and $12n_{725}$\footnote{\change{After this paper was first appeared, the author and Zhou found an infinite family of $(1,1)$ knots with $\CFK^\infty$ of type (A)~\cite{binns20231}}.}. There are many almost staircase complexes which do not arise as the complexes of $(2,4g(K)-3)$-cables of $L$-space knots or as box plus staircase complexes. The author is unaware if any such complex arises as $\CFK^\infty(K)$ for some $K$. Indeed, the author is likewise unaware of a classification of staircase complexes realised as $\CFK^\infty(K)$ for some $K$.

We prove Theorem~\ref{infinityclassification} by adapting the techniques Ozsv\'ath-Szab\'o use in classifying the $\CFK^\infty$ type of $L$-space knots. With Theorem~\ref{infinityclassification} at hand we can obtain the following result:

\newtheorem*{stability}{Proposition~\ref{stability}}
\begin{stability}
Let $K$ be a knot of genus $g$. $K$ is an almost $L$-space knot if and only if $\rank(\widehat{\HF}(S^3_{p/q}(K)))=p+2q$ for all $p/q\geq 2g-1$.
\end{stability}

This is an analogue of the result that if $K$ is a knot of genus $g$ then $K$ is an $L$-space knot if and only if $\rank(\widehat{\HF}(S^3_{p/q}(K)))=p$ for all $p/q\geq 2g-1$. We prove this result with the aid of \change{Hanselman}-Rasmussen-Watson's immersed curve interpretation of the surgery formula in bordered Floer homology~\cite{hanselman2016bordered}.

$\CFK^\infty(-)$ comes equipped with a grading called the Maslov grading. The differential lowers the Maslov grading by $1$ while the $U$ action decreases it by $2$. There is a weaker version of $\CFK^\infty(K)$, denoted $\widehat{\HFK}(K)$, that carries only a Maslov grading and an Alexander grading. Theorem~\ref{infinityclassification} can be sharpened somewhat via the following proposition;

\newtheorem*{prop:delta0}{Proposition~\ref{prop:delta0}}
\begin{prop:delta0}
  Suppose $K$ is an almost $L$-space knot with the $\CFK^\infty$ type of a box plus staircase complex. Then $\widehat{\HFK}(K,0)$ is supported in a single Maslov grading.
\end{prop:delta0}

 This proposition is proven using involutive knot Floer homology~\cite{hendricks2017involutive} and Sarkar's basepoint moving map~\cite{sarkar2015moving}. This result can perhaps be though of as a relative version of a result of Hanselman-Kutluhan-Lidman concerning the geography problem for $\HF^+(Y,\mathfrak{s})$ -- another version of Heegaard Floer homology -- for $\mathfrak{s}$ a self conjugate $spin^c$ structure, with $\widehat{\HF}(Y,\mathfrak{s})$ of next to minimal rank~\cite{hanselman2019remark}.

A number of topological properties of $L$-space knots can be deduced from the classification of their $\CFK^\infty$-type -- for example that they are strongly-quasi-positive and fibered. Likewise Theorem~\ref{infinityclassification} allows us to show that almost $L$-space knots satisfy or almost satisfy various strong topological properties.

\newtheorem*{genus1class}{Corollary~\ref{genus1class}}
\begin{genus1class}[\cite{baldwin2022characterizing}]
Suppose $K$ is a genus one almost $L$-space knot. Then $K$ is $T(2,-3)$, the figure eight knot or  the mirror of the knot $5_2$.
\end{genus1class}

This result is the analogue of the fact that the only genus one $L$-space knot is $T(2,3)$.

\newtheorem*{comp}{Corollary~\ref{comp}}
\begin{comp}
The only composite almost $L$-space knot is $T(2,3)\# T(2,3)$.
\end{comp}

This is the analogue of \change{ Krcatovich's result that $L$-space knots are prime~\cite[Theorem 1.2]{krca}. See also~\cite[Corollary 1.4]{baldwin_note_2018}}.

\newtheorem*{fibered}{Corollary~\ref{fibered}}
\begin{fibered}[\cite{baldwin2022characterizing}]
The mirror of $5_2$ is the only almost $L$-space knot which is not fibered.
\end{fibered}

This is the analogue of the fact that $L$-space knots are fibered.

\newtheorem*{tauk}{Corollary~\ref{tauk}}
\begin{tauk}
Suppose $K$ is an almost $L$-space knot for which ${|\tau(K)|< g_3(K)}$. Then $K$ is the figure eight knot.
\end{tauk}

Here $\tau(-)$ is a knot invariant due to Ozsv\'ath-Szab\'o~\cite{ozsvath2004holomorphic}. Note that ${|\tau(K)|\leq g_3(K)}$ for all $K$. This result is the analogue of the fact that for $L$-space knots $K$, ${\tau(K)=g_3(K)}$.
\begin{corollary}[\cite{baldwin2022characterizing}]
      The only almost $L$-space knots that are not strongly quasi-positive are the figure eight knot and the left handed trefoil.
\end{corollary}

This is the analogue of the fact that $L$-space knots are strongly quasi-positive.

In a different direction, Theorem~\ref{infinityclassification} allows us to recover a result of Hedden:

\newtheorem*{Hedd}{Corollary~\ref{Hedd}}
\begin{Hedd}[\cite{hedden2007knot,hedden2011floer}]
    Suppose $K$ is a knot and $\rank(\widehat{\HFK}(K_n))=n+2$. Then $K$ is a non-trivial $L$-space knot and $n=2g(K)-1$.
\end{Hedd}

Here $K_n$ indicates the core of $n$-surgery on $K$. We again prove this with the aid of Hanselman-Rasmussen-Watson's immersed curve invariant. From here we can obtain some link detection results for cables of $L$-space knots using the same arguments as were used by the author and Dey for similar purposes in~\cite{binns2022cable}. Let $K_{p,q}$ denote the $(p,q)$-cable of $K$. Here $p$ indicates the longitudinal wrapping number while $q$ indicates the meridional wrapping number. We have the following:

\newtheorem*{HFKcables}{Proposition~\ref{HFKcables}}
\begin{HFKcables}

      Suppose $K$ is an $L$-space knot. If a $2$-component link $L$ satisfies $\widehat{\HFK}(L)\cong\widehat{\HFK}(K_{2,4g(K)-2})$. Then $L$ is a $(2,4g(K)-2)$-cable of an $L$-space knot $K'$ such that $\widehat{\HFK}(K')\cong\widehat{\HFK}(K)$.
 
\end{HFKcables}

Here we orient $(2,2n)$-cables as the boundary of annuli. Moreover we have a stronger result for \emph{link Floer homology}, an invariant due to Ozsv\'ath-Szab\'o~\cite{HolomorphicdiskslinkinvariantsandthemultivariableAlexanderpolynomial}:

\newtheorem*{HFLcables}{Proposition~\ref{HFLcables}}
\begin{HFLcables}
    
    Suppose $K$ is a $L$-space knot and $L$ satisfies $\widehat{\HFL}(L)\cong\widehat{\HFL}(K_{m,2mg(K)-m})$. Then $L$ is a $(m,2mg(K)-m)$-cable of an $L$-space knot $K'$ with $\widehat{\HFK}(K')\cong\widehat{\HFK}(K)$.
\end{HFLcables}

In fact, the conclusion that $K'$ is an $L$-space knot in the above two Propositions follows from the fact that $\widehat{\HFK}(K)\cong\widehat{\HFK}(K')$, as is also true for the corresponding results in~\cite{binns2022cable}. This observation follows either from examining the immersed curve invariant or a result of Lidman-Moore-Zibrowius~\cite[Lemma 2.7]{lidman2020space}.

As a corollary we have the following detection results:
\begin{corollary}
    Knot Floer homology detects $T(2,3)_{2,2}$ and $T(2,5)_{2,6}$ amongst two component links.
\end{corollary}
Here these two links are again oriented as the boundary of annuli. The corresponding result for Link Floer homology is the following:
\begin{corollary}
Link Floer homology detects  $T(2,3)_{m,2m}$ and $T(2,5)_{m,6m}$.
\end{corollary}

Note that $T(2,3)$ and $T(2,5)$ are $L$-space knots. The above two results thus follow directly from Propositions~\ref{HFKcables} and ~\ref{HFLcables}, Ghiggini's result that knot Floer homology detects $T(2,3)$~\cite{ghiggini2008knot}, and Farber-Reinoso-Wang's result that knot Floer homology detects $T(2,5)$~\cite{farber2022fixed}.

We also give a characterization of links with the same link Floer homology type as the $(m,mn)$-cables of almost $L$-space knots for large $n$, generalizing work of the author and Dey~\cite[Theorem 3.1]{binns2022cable}:

\newtheorem*{almostLspacecables}{Proposition~\ref{almostLspacecables}}
\begin{almostLspacecables}
  Let $K$ be an almost $L$-space knot.  Suppose $L$ is a link such that $\widehat{\HFL}(L)\cong \widehat{\HFL}(K_{m,mn})$ with $\change{n}>2g(K)-1$. Then $L$ is the $(m,mn)$-cable of an almost $L$-space knot $K'$ such that $\widehat{\HFK}(K')\cong\widehat{\HFK}(K)$.
\end{almostLspacecables}

By combining the above result and Corollary~\ref{genus1class} we can also deduce the following:

\newtheorem*{genus1almostLspacecables}{Proposition~\ref{genus1almostLspacecables}}
\begin{genus1almostLspacecables}
    Link Floer homology detects the $(m,mn)$-cables of $T(2,-3)$, the figure eight knot and the mirror of $5_2$ for $n>1$.
\end{genus1almostLspacecables}

We conclude this introduction with some questions:

\begin{question}
Is there a version of the $L$-space conjecture for almost $L$-spaces?
\end{question}

The $L$-space conjecture posits an equivalent characterization of $L$-spaces in terms of both an orderability condition on their fundamental groups, and whether or not they admit taut foliations~\cite{boyer2013spaces}.

Given a property that $L$-space knots exhibit, it is natural to ask if almost $L$-space knots too exhibit that property. For example, $L$-space knots do not have essential Conway spheres by a result of Lidman-Moore-Zibrowius~\cite{lidman2020space}, so it is natural to ask the following:

\begin{question}
    Do almost $L$-space knots have essential Conway spheres?
\end{question}

This paper is organised as follows. In Section~\ref{review} we review aspects of the Heegaard Floer homology package relevant to subsequent sections. In Section~\ref{remarkssection} we give some alternate characterizations of almost $L$-spaces and prove Proposition~\ref{stability} and Proposition~\ref{prop:delta0}. In Sections~\ref{nomaslov}, \ref{maslov}, and~\ref{chain} we prove Theorem~\ref{infinityclassification}. In Section~\ref{applications} we give the applications.

\subsection*{\change{Acknowledgments}}
The author would like to thank his advisor John Baldwin for numerous helpful conversations as well as Jen Hom, Mihai Marian, Allison Moore and Liam Watson. He would like to thank Subhankar Dey and Duncan McCoy for some useful comments on an earlier draft. He is especially grateful to Hugo Zhou for explaining how to understand a key step in Section~\ref{chain}, as well as Francesco Lin for drawing his attention to Hanselman-Lidman-Kutluhan's result~\cite{hanselman2019remark}, which lead him to the proof of Proposition~\ref{prop:delta0}. \change{Finally, he would like to thank the referee for their patience and helpful feedback which have substantially improved this paper.}

\section{A Review of Heegaard Floer Homology}\label{review}

In this section we review background material on Heegaard Floer homology with an emphasis on aspects of the theory which will play a role in subsequent sections.

\subsection{Heegaard Floer homology for closed $3$-manifolds}

 \emph{Heegaard Floer homology}, denoted $\widehat{\HF}(-)$, is an invariant of three manifolds due to Ozsv\'ath-Szab\'o. It is defined using Heegaard diagrams, symplectic topology and analysis~\cite{ozsvath2004holomorphic}. It can be defined with integer coefficients but we will take coefficients in $\change{\F:=\Z/2}$ throughout this paper. $\widehat{\HF}(Y)$ splits as a direct sum over equivalence classes of non-vanishing vector fields on $Y$ called \emph{$\spin^c$-structures}. Let $Y$ be a rational homology sphere. There is a non-canonical bijection between the set of $\spin^c$-structures on $Y$ and $H_1(Y;\Z)$. Suppose $Y$ is given by performing $n$-surgery on a knot $K$ in $S^3$. Then $K$ determines a canonical bijection between $\spin^c(Y)$ and $\Z/n\cong H_1(Y;\Z)$ as follows. Observe that the trace of $n$\change{-}surgery, $X_n(K)$, gives a cobordism from $S^3_n(K)$ to $S^3$. Fix a Seifert surface $\Sigma$ for $K$. Consider the surface $\widehat{\Sigma}$ obtained by capping off $\Sigma$ in $X_n(K)$. Suppose $\mathfrak{s}$ is a $\spin^c$ structure over $S^3_n(K)$ that admits an extension $\mathfrak{t}$ over $X_n(K)$ with the property that $\langle c_1(\mathfrak{t}),[\widehat{\Sigma}]\rangle -n\equiv 2i\mod{2n}$. Then the map $\mathfrak{s}\mapsto i$ is an isomorphism by a result of Ozsv\'ath-Szab\'o~\cite[Lemma 2.2]{ozsvath2008knot}.

 \subsection{Knot Floer homology}
 There is a version of Heegaard Floer homology for knots in $S^3$ called \emph{knot Floer homology}. This invariant is due independently to Ozsv\'ath-Szab\'o~\cite{Holomorphicdisksandknotinvariants} and J. Rasmussen~\cite{Rasmussen}. The strongest version of knot Floer homology is denoted $\CFK^\infty(-)$. This form of the invariant takes value in the category of finitely generated filtered chain complexes of $\F[U,U^{-1}]$-modules up to filtered chain homotopy equivalence. We will often conflate the $\CFK^\infty$-type of such a complex and the associated graded \change{equipped} with the induced differential. These complexes carry a grading called the \emph{Maslov grading}. The $U$ action lowers the Maslov grading by $2$, and induces another grading on $\CFK^\infty(-)$, called the $U$ grading. The subcomplex generated by generators of $U$-grading $0$ is denoted $\widehat{\CFK}(K)$, and has homology denoted by $\widehat{\HFK}(K)$. The other grading is called the \emph{Alexander grading}.

  Ozsv\'ath-Szab\'o showed that $\CFK^\infty(K)$ determines $\widehat{\HF}(S^3_q(K))$ for all $q\in\Q$~\cite{ozsvath2010knot}.

  \subsection{Chain Complexes for almost $L$-space knots}

  Theorem~\ref{infinityclassification} states that the chain complexes of almost $L$-space knot admit particularly simple models. That is, there exist particularly simple models of their chain homotopy type. Throughout this paper we shall use $i$ to indicate the $U$-grading on $\CFK^\infty$ and $j$ to indicate the Alexander grading. To make the statement of Theorem~\ref{infinityclassification} we need the following definitions:
 
\begin{definition}\label{staircase}
A \emph{staircase complex} is a set of \change{homogeneous} generators \change{$\{ x_k\}_{1\leq k \leq N} \cup \{ y_k\}_{1\leq k \leq N+1}$} where $N\in\Z^{\geq 0}$ such that $\change{x_k}$ and $\change{y_k}$ only differ in $i$ coordinate while $\change{x_k}$ and $\change{y_{k+1}}$ only differ in $j$ coordinate, and the non-trivial \change{components of the} differential are given by $\partial \change{x_k = y_k + y_{k+1}}$ for $  1\leq \change{k}\leq N.  $ \end{definition}

An example of such a complex is shown in Figure~\ref{staircasefig}.

\begin{definition} \label{def:almoststaircase1}
   Let $N\in\Z^{\geq 0}$. A \emph{type $1$ almost staircase complex} is a complex admitting a \change{homogeneous} basis $\change{\{x_{k}\}_{0<|k|\leq N+1}\cup\{y_k\}_{0< |k|\leq N}}\cup\{z\}$ such that:\begin{enumerate}
       \item For $k>0$, $x_k$ and $y_{k-1}$ only differ in $\change{j}$ coordinate while $\change{x_k}$ and $\change{y_{k}}$ only differ in $\change{i}$ coordinate. For $k<0$, $x_k$ and $y_{k}$ only differ in $i$ coordinate while $\change{x_{k+1}}$ and $\change{y_{k}}$ only differ in $j$ coordinate.
  
       \item Suppose $x_{-1}$ is of bigrading $(a,b)$. Then $z$ is of bigrading $(a-1,b)$ and $x_{1}$ is of bigrading $(a-1,b+1)$.

       \item   The non-trivial components of the differential are given by: \begin{enumerate}
           \item \change{$\partial y_k = x_k + x_{k+1}$ for $k\neq \pm1$},
           \item $\partial y_{\pm1}=x_{\pm 2}+x_1+x_{-1}$,
           \item $\partial x_{\pm1}=z$.
        
       \end{enumerate}
   \end{enumerate}
An example of such a complex is shown in Figure~\ref{almoststaircase2fig}. The $N=0$ case corresponds to the case of the left handed trefoil.
   
\end{definition}

\begin{definition} \label{def:almoststaircase2}
   Let $N\in\Z^{\geq 1}$. A \emph{type $2$ almost staircase complex} is a complex admitting a \change{homogeneous} basis $\change{\{x_{i}\}_{0<|k|\leq N}\cup\{y_k\}_{0<|k|\leq N}\cup\{z\}}$ such that:\begin{enumerate}
       \item For \change{$k<0$, $x_k$ and $y_{k}$ only differ in $i$ coordinate while $x_{k-1}$ and $y_{k}$} only differ in $j$ coordinate. For \change{$k>0$, $x_k$ and $y_{k}$ only differ in $j$ coordinate while $x_{k}$ and $y_{k+1}$} only differ in $i$ coordinate
       \item Suppose $y_{-1}$ is of bigrading $(a,b)$. Then $z$ is of bigrading $(a,b+1)$ and $y_1$ is of bigrading $(a-1,b+1)$.

       \item  The non-trivial components of the differential are given by: \begin{enumerate}
           \item \change{$\partial y_k = x_k + x_{k+1}$ for $k<\change{-1}$,
           \item $\partial y_k = x_k + x_{k-1}$ for $k>\change{1}$},
           \item $\partial y_{\pm 1}=x_{-1}+x_1$,
           \item $\partial z=\change{y_{-1}}+y_1$.
        
       \end{enumerate}
   \end{enumerate}

\end{definition}
An example of such a complex is shown in Figure~\ref{almoststaircase1fig}. Almost staircases of type $1$ arise as $(2,4g(K)-3)$-cables of $L$-space knots while almost staircase complexes of type $2$ do not. This can be deduced using Hanselman-Watson's cabling formula for knot Floer homology in terms of immersed curves~\cite{hanselman2019cabling}. Indeed, the author is unaware of examples of almost staircase complexes of type $2$ which arise as the knot Floer homology of some knot \change{which are not also staircase plus box complexes}.

\begin{definition}\label{def:box}

A \emph{box complex} is a complex with of generators $x_3,x_4,x_1,x_2$ with coordinates $(0,0), (0,1), (1,0)$ and $(1,1)$ respectively, up to overall shift in bigrading,  differentials given by:\begin{enumerate}
    \item$ \partial x_2 = x_1 + x_4 \hspace{5em}$\change{.}
    \item $\partial \change{x_1} = \partial x_4 = x_3$.
    \item $0$ otherwise.
\end{enumerate}
  
\end{definition}
An example of such a complex is shown in Figure~\ref{boxfig}. With these definitions at hand, the statement of Theorem~\ref{infinityclassification} is made rigorous. We note \change{again, however,} that there is a degree of overlap between \change{complexes with the filtered chain homotopy type of box plus staircase complexes and almost staircase complexes}. Namely, in light of Proposition~\ref{prop:delta0}, box plus staircase complexes for knots $K$ with $\rank(\widehat{\HFK}(K,1))=2$ can arise as almost staircase complexes. Note that the mirror of $5_2$, $T(2,3)\#T(2,3)$, $10_{139}$ and $12n_{725}$ all satisfy this property.

\section{Some Remarks on Almost $L$-spaces and almost $L$-space knots}\label{remarkssection}

In this section we give some equivalent definitions of almost $L$-spaces and almost $L$ space knots.

First we give a definition of almost $L$-spaces in terms of of $\HF^+$\change{.}

\begin{proposition}
Let $Y$ be a rational homology sphere. The following conditions are equivalent:
\begin{enumerate}
    \item  $Y$ is an almost $L$-space.
    \item $\rank(\widehat{\HF}(Y,\mathfrak{s}))=1$ aside from for a unique $\spin^c$ structure $\mathfrak{s}_0$ for which $\rank(\widehat{\HF}(Y,\mathfrak{s}_0))=3$.
    \item $\HF^+(Y,\mathfrak{s})\cong\tau^+$ aside from in a unique $\spin^c$ structure $\mathfrak{s}_0$ for which $\HF^+(Y,\mathfrak{s}_0)\cong\tau^+\oplus\change{\F[U]/(U^n)}$ for some $n$.
    \item $\HF^+_{\text{red}}(Y,\mathfrak{s})\cong0 $ aside from in a unique $\spin^c$ structure $\mathfrak{s}_0$ for which $\HF^+_{\text{red}}(Y,\mathfrak{s}_0)\cong\change{\F[U]/(U^n)}$ for some $n$.
\end{enumerate}
\end{proposition}
Here $\tau^+$ indicates a tower $\F[U,U^{-1}]/\change{(U)}$ and $\HF^+_\text{red}(Y,\mathfrak{s})$ is the submodule of $\change{\HF^+}(Y,\mathfrak{s})$ generated by elements that are not in the image of $U^n$ for sufficiently large $n$. The proof is routine.
\begin{proof}
$1\iff 2$ is just our definition of almost $L$-space.

To see that $2\implies 3$, note that in general $\HF^+(Y,\mathfrak{s})=\tau^+\oplus \HF^+_\text{red}(Y,\mathfrak{s})$. Now  considering the long exact sequence induced by the short exact sequence:

\[
\begin{tikzcd}
  0 \arrow[r] & \widehat{\CF}(Y,\mathfrak{s})  \arrow[r, ""] & \CF^+(Y,\mathfrak{s}) \arrow[r, "U"] & \CF^+(Y,\mathfrak{s}) \arrow[r] & 0 
\end{tikzcd}
\]

we see that $\HF^+(Y,\mathfrak{s})\cong\tau^+$ unless $\mathfrak{s}=\mathfrak{s}_0$. In the latter case observe that for each summand of $\HF^+_\text{red}(Y,\mathfrak{s})$ we find two generators of $\widehat{\HF}(Y,\mathfrak{s}_0)$. It follows that $\HF^+_\text{red}(Y,\mathfrak{s}_0)\cong \Z[U]/\change{(U^n)}$ for some $n\geq 0$.
$3\implies 4$ follows from the definition. $4\implies 2$ follows immediately from the long exact sequence used above.\end{proof}

A similar statement can be made for $\HF^-$:
\begin{proposition}
Let $Y$ be a rational homology sphere. The following conditions are equivalent:
\begin{enumerate}
    \item  $Y$ is an almost $L$-space.

    \item $\HF^-(Y,\mathfrak{s})\cong\tau^-$ aside from in a unique $\spin^c$ structure $\mathfrak{s}_0$ for which $\HF^-(Y,\mathfrak{s}_0)\cong\tau^+\oplus\change{\F[U]/(U^n)}$ for some $n$.
    \item $\HF^-_{\text{red}}(Y,\mathfrak{s})\cong0 $ aside from in a unique $\spin^c$ structure $\mathfrak{s}_0$ for which $\HF^-_{\text{red}}(Y,\mathfrak{s}_0)\cong\change{\F}[U]/(U^n)$ for some $n$.
\end{enumerate}
\end{proposition}
Here $\tau^-$ indicates a tower $\F[U]$ and $\HF^-_{\text{red}}(Y,\mathfrak{s})$ is the $U$ torsion submodule of $\HF^-(Y,\mathfrak{s})$. The proof is again routine.
\begin{proof}
This follows just as in the case of the previous proposition but applying the following short exact sequence:

\[
\begin{tikzcd}
  0 \arrow[r] & \CF^-(Y,\mathfrak{s})  \arrow[r, "U"] & \CF^-(Y,\mathfrak{s}) \arrow[r, ""] & \widehat{CF}(Y,\mathfrak{s}) \arrow[r] & 0 .
\end{tikzcd}
\]\end{proof}

\begin{remark}
Note that one might reasonably have given a stronger definition of almost $L$-spaces as rational homology spheres with, say, $\HF^+(Y;\change{\F})\cong\tau^+\oplus\F$. It is not clear to the author if this would be a better definition.\footnote{\change{It turns out that the distinction between these two notions is intimately tied to the $L$-space conjecture; see~\cite{lin2023remark}, which appeared subsequently to this article.}}
\end{remark}

 We now give the following alternate characterisation of almost $L$-space knots:

\begin{proposition}\label{stability}
Let $K$ be a knot of genus $g$. $K$ is an almost $L$-space knot if and only if $\rank(\widehat{\HF}(S^3_{p/q}(K)))=p+2q$ for all $p/q\geq 2g-1$ \change{with $p,q>0$}.
\end{proposition}
We find it convenient to prove this proposition by way of Hanselman-Rasmussen-Watson's immersed curve interpretation of bordered Floer homology~\cite{hanselman2016bordered,hanselman2018heegaard}. This invariant can be thought of as assigning to each knot $K$ a closed multi-curve  $\Gamma(K)$ in the infinite strip $[0,1]\times \R$, punctured at the points $\{0\}\times\Z$. Pulling $\gamma(K)$ tight\change{,} we obtain a collection of straight line segments. This is called a~\emph{singular pegboard diagram} in~\cite{hanselman2016bordered}. There is a unique such segment which is not vertical. Let $m$ denote its slope. Let $n$ be the number of vertical line segments, counted with multiplicity. Note that $n-m$ is even. We will apply the following proposition of Hanselman:

\begin{proposition}[{{Hanselman~\cite[\change{Proposition 15}]{hanselman2022heegaard}}}]\label{Hanselmanrank}
Let $K$ be a knot with $m$ and $n$ as above $${\rank(\widehat{\HF}(S^3_{p/q}(K)))=|p-qm|+n|q|}.$$
\end{proposition}

We now prove Proposition~\ref{stability}.

\begin{proof}[Proof of Proposition~\ref{stability}]
We first prove the forward direction. If $K$ is an almost $L$-space knot, Theorem~\ref{infinityclassification} implies that $m=2g(K)-1$, $n=2g(K)+1$. \change{These two facts can be deduced from ~\cite[Proposition 48]{hanselman2018heegaard}, which shows how to produce the immersed curve invariant of the exterior of a knot $K$ given a horizontally and vertically simplified basis for  $\CFK^\infty(K)$. Note that each filtered basis in the statement of Theorem~\ref{infinityclassification} is both horizontally and vertically simplified (see Definition~\cite[Definition 11.23]{LOTborderedbook}.)} Proposition~\ref{Hanselmanrank} immediately implies that ${rank(\widehat{\HF}(S^3_{p/q}(K)))=p+2q}$ for all $p/q\geq 2g(K)-1$.

For the opposite direction\change{,} suppose that there exist $p,q$ such that $\frac{p}{q}\geq 2g(K)-1$ and $\rank(\widehat{\HF}(S^3_{p/q}(K)))=p+2q$. We first show that $K$ is an $L$-space knot or an almost $L$-space knot. Note that Proposition~\ref{Hanselmanrank} implies that, with $m,n$ as above $p+2q=n|q|+|mq-p|$. Since $K$ cannot be the unknot $U$ as $\rank(\widehat{\HF}(S^3_{p/q}(U)))=p$, we have that $2g(K)-1>0$ and may take  $p,q\geq 0$, so that $p+2q=nq+|mq-p|$. Suppose $m> \frac{p}{q}$. Then $2\dfrac{p}{q}+2=n+m$, so that $2m+2\geq n+m$ and $m+2> n$. But $n-m\in 2\Z^{\geq 0}$, so that $n=m$ in which case $K$ is an $L$-space knot and $\frac{p}{q}=m-1$. Otherwise we have that $n=m+2$. It follows in turn that $\rank(\widehat{\HF}(S^3_s(K)))=s+2$ for any sufficiently large $s\change{\in\Z}$ so that $K$ is an almost $L$-space knot.\end{proof}

We conclude this section with the following proposition, taking as given Theorem~\ref{infinityclassification}:

\begin{proposition}\label{prop:delta0}
     Suppose $K$ is an almost $L$-space knot with the $\CFK^\infty$-type of a box plus staircase complex. Then $\widehat{\HFK}(K,0)$ is supported in a single Maslov grading.
\end{proposition}

\change{Note that in combination with Theorem~\ref{infinityclassification}, this proposition implies that the Alexander polynomial of an almost $L$-space knot $K$ determines $\CFK^\infty(K)$. We will use a weaker version of this statement --- that the Alexander polynomial of an almost $L$-space knot $K$ determines $\widehat{\HFK}(K)$ --- in the proof of Proposition~\ref{almostLspacecables}.}

\begin{proof}

Suppose $K$ is an almost $L$-space knot with the $\CFK^\infty$-type of a box plus staircase complex.

Consider Hendricks-Manolescu's involution $\iota$ on $\CFK^\infty(K)$~\cite{hendricks2017involutive}. Observe that $\iota^2$ yields Sarkar's basepoint pushing map~\cite{sarkar2015moving}. Thus $\iota^2$ is determined up to filtered chain homotopy by the $\CFK^\infty$-type of $K$ by a result of Zemke~\cite[Corololary C]{zemke2017quasistabilization}. Consider the components of $\iota$ and Sarkar's basepoint pushing map acting on the box complex. Endow the box complex with the generators $x_1,x_2,x_3,x_4$ as in Definition~\ref{def:box}. Let $e$ denote $U^{-1}x_3$. Sarkar's basepoint pushing map is given by $x_1\mapsto x_1$, $x_4\mapsto x_4$, $e\mapsto e$ and $x_2\mapsto x_2+e$. \change{Let $z$ be the generator of $\CFK^\infty(K)$ of Alexander grading zero and $U$-grading $0$ that is contained in the staircase.} Suppose that $z$ is not of the same Maslov grading as $e$ and $x_2$. Observe that the component of $\iota(x_2)$ in $(U,A)$-grading $(0,0)$ is given by either $e+x_2$ or $x_2$, while $\iota(e)=e$. Neither of these maps square to Sarkar's basepoint pushing map, so we have a contradiction and that $z$ is of the same Maslov grading as $e$ and $x_2$.\end{proof}

\section{$\CFK^\infty$ of almost $L$-space knots as a bigraded vector space without Maslov gradings}\label{nomaslov}

In this sectionm we determine $\CFK^\infty$ as a bigraded vector space for almost $L$-space knots without Maslov grading. We apply similar techniques as used by  Ozsv\'ath-Szab\'o~\cite{ozsvath2005knot} in the proof of the corresponding result for $L$-space knots.

Recall that $\widehat{\HF}(S^3_n(K))$ admits \change{a grading by $\spin^c$-structures on $S^3_n(K)$, which are in turn in bijection with }$\Z/n$. We first determine the grading in which $\rank(\widehat{\HF}(S^3_n(K),[i]))=3$.

\begin{lemma}\label{spin0}
    Suppose $K$ is an almost $L$-space knot and $n\geq2g(K)-1$ is an integer. Then $\rank(\widehat{\HF}(S^3_n(K)),[i])$ is $1$ unless $[i]=0$, in which case it is $3$.
\end{lemma}

\begin{proof}
    The set of $\spin^c$ structures admit\change{s} a conjugation action which induces an isomorphism on $\widehat{\HF}(S^3_n(K))$. This action sends $\widehat{\HF}(S^3_n(K),[i])$ to $\widehat{\HF}(S^3_n(K),[-i])$. Since there is a unique $\spin^c$ structure in which $\widehat{\HF}(S^3_n(K))$ is not one, it follows that this $\spin^c$ structure is self conjugate. There are potentially two such $\spin^c$ structures, namely $0$ and $\frac{n}{2}$ if $n$ is even. We show that the latter case is impossible. Since $n\geq 2g(K)-1$ we have that $\frac{n}{2}\geq g(K)$. However, it follows immediately from Ozsv\'ath-Szab\'o's surgery formula~\cite{ozsvath2008knot} that $\widehat{\HF}(S^3_n(K);[i])$ is of rank $1$ for $i\geq g(K)$, a contradiction.\end{proof}

We note in passing that if $K$ is an $L$-space knot then

$${\rank(\widehat{\HF}(S^3_{2g(K)-2}(K),[g(K)-1])=3}.$$

Before proceeding\change{,} we introduce some notation. $\CFK^\infty$ is a $\Z\oplus\Z$ graded complex, where the grading $(i,j)$ indicates a $U$-grading of $j$ and an Alexander grading of $i$. \change{We let $\{\mathbf{A}\}$ denote the set of generators satisfying condition $\mathbf{A}$, whatever condition $\mathbf{A}$ may be}. We denote by $C\change{(S)}$ the (\change{appropriate quotient of the}) subcomplex of the associated graded consisting of generators in the set \change{ $S\subset\Z\oplus\Z$, for appropriate $S$. Likewise\change{,} let $H\change{(S)}$ denote the homology of $C(S)$.}

 Set $X_m=\{i\leq 0,j=m\}$, $Y_m=\{i=0,j\leq m-1\}$. Let $UX_m$ be the \change{set} $\{i< 0,j=\change{m}-1\}$.

Following, Ozsv\'ath-Szab\'o's approach in the $L$-space knot setting~\cite[Section 3]{ozsvath2005knot}, observe that we have a pair of short exact sequences:

\begin{center}

\begin{tikzcd}
&&&0\arrow[d]\\

0\arrow[r]&C(UX_m\change{)}\arrow[r]&C\change{(}UX_m\cup Y_m\change{)}\arrow[r]& C(Y_m\change{)}\arrow[d]\arrow[r]&0\\

&&&C(X_m\cup Y_m\change{)}\arrow[d]\\&&&C(X_m\change{)}\arrow[d]\\&&&0
\end{tikzcd}

\end{center}

This yields a pair of exact triangles in homology shown in Figure~\ref{exacttriangles}. As noted by Ozsv\'ath-Szab\'o~\cite[Section 3]{ozsvath2005knot}, the composition of the two horizontal maps is zero.

\begin{figure}[h]

\begin{tikzcd}
H_*(UX_m)\arrow[dr,"e"]&&H_*(Y_m)\arrow[dr,"f"]\arrow[ll,"c"]&&H_*(X_m)\arrow[ll,"d"]\\
&H_*(UX_m\cup Y_m)\arrow[ur,"b"]&&H_*(X_m\cup Y_m)\arrow[ur,"a"]
\end{tikzcd}

\caption{A pair of exact triangles we will use repeatedly in this section and the next.}\label{exacttriangles}
\end{figure}

We first extract the following lemma from work of Ozsv\'ath-Szab\'o~\cite[Section 3]{ozsvath2005knot}:
   
\begin{lemma} \label{alexander1}
Suppose $H_*(X_m\cup Y_m)\cong\F$, $H_*(UX_m\cup Y_m)\cong \F^3$. Then $1\leq \rank(H_*(X_m))\leq 2$.

\end{lemma}

\begin{proof}

Suppose $H_*(X_{\change{m}})\cong \F^n$ \change{for some integer $n$}. In this specific context the exact triangles from Figure~\ref{exacttriangles} yield:
\begin{center}
\begin{tikzcd}
    \F^n\arrow[dr]&&H_*(Y_{\change{m}})\arrow[dr]\arrow[ll,"c"]&&\F^n\arrow[ll,"d"]\\
&\F^3\arrow[ur, "b"]&&\F \arrow[ur, "a"]
\end{tikzcd}

\end{center}

If $a$ is injective then $H_*(Y_{\change{m}})\cong \F^{n-1}$, $d$ has kernel of rank $1$, $b$ must have image of rank $1$, and $c$ must have kernel of rank $1$. Since the composition of the two horizontal maps must be trivial it follows that $1\leq n\leq 2$.

If $a$ is trivial then $d$ is injective and $H_*(Y_{\change{m}})\cong \F^{n+1}$ whence $b$ has image of rank $2$ and $c$ has kernel of rank $2$. It follows that the image of $c\circ d$ is of rank $n-1$ or $n-2$, whence $n=1$ or $n=2$ since the composition of the two maps is trivial.\end{proof}

\begin{lemma}\label{alexander0}

Suppose $H_*(X_m\cup Y_m)\cong\F^3$ \change{and} $H_*(UX_m\cup Y_m)\cong \F$. Then $1\leq \rank(H_*(X_m))\leq 2$.

\end{lemma}

\begin{proof}
This follows from the proof of the previous lemma, after dualizing, noting that we are working over a field.\end{proof}

Of course we are interested in computing $H_*(\change{\{(}0,m\change{)\}})$ rather than $H_*(X_m)$. We now relate these quantities.

\begin{lemma}\label{Xtom}
We have the following:\begin{enumerate}
    \item Suppose $\rank(H_*(X_m))=2$. Then $\rank(H_*(\change{\{(}0,m)\change{\})})$ is one of $\rank(H_*(\{i<0,j=m\}))-2,\rank(H_*(\{i<0,j=m\}))$ or $\rank(H_*(\{i<0,j=m\}))+2$.
    \item Suppose $\rank(H_*(X_m))=1$. Then $\rank(H_*(\change{\{(}0,m\change{)\}}))$ is either $\rank(H_*(\{i<0,j=m\}))-1$ or $\rank(H_*(\{i<0,j=m\}))+1$.
    \item Suppose $\rank(H_*(X_m))=0$. Then $\rank(H_*(\change{\{(}0,m\change{)\}}))=\rank(H_*(\{i<0,j=m\}))$.
\end{enumerate}
\end{lemma}

\begin{proof}
There is a short exact sequence:
\begin{center}
\begin{tikzcd}
    0\arrow[r]&C\change{(}\{i<0,j=m\}\change{)}\arrow[r]&C\change{(}\{i\leq 0,j=m\}\change{)}\arrow[r]&C(\change{\{(}0,m\change{)\}})\arrow[r]&0
\end{tikzcd}

\end{center}
giving the exact triangle on homology shown in Figure~\ref{exacttriangle}.
\begin{figure}[h]
\begin{tikzcd}
    H_*\change{(}\{i<0,j=m\}\change{)}\arrow[rr,"a"]&&H_*\change{(}\{i\leq 0,j=m\}\change{)}\arrow[dl,"b"]\\&H_*\change{(\{(}0,m\change{)\}})\arrow[ul,"c"]
\end{tikzcd}
\caption{Another exact triangle we will use repeatedly in this section and the next.}\label{exacttriangle}
\end{figure}

The result follows directly.\end{proof}

We proceed now to compute the rank of $H_*(\change{\{(}0,n\change{)\}})$ -- i.e. $\widehat{\HFK}(K,n)$ -- for each $n$.

\begin{lemma}\label{LemmaX1}
Suppose $K$ is an almost $L$-space knot. Then $\rank(\widehat{\HFK}(K,A))\leq 1$ for $|A|>1$.
\end{lemma}

\begin{proof}
Given Lemma~\ref{spin0}, this follows directly from work of Ozsv\'ath-Szab\'o in the setting in which $K$ is an $L$-space knot as opposed to an almost $L$-space knot~\cite[Section 3]{ozsvath2005knot}.
\end{proof}

Moreover we have the following:

\begin{lemma}\label{oddeven}
    Suppose $K$ is an almost $L$-space knot. Then there is a $d\in 2\Z$ such that:\begin{enumerate}
        \item $H_*(X_2)\cong 0$, $H_*(Y_2)\cong\F_d$ or;
        \item $H_*(X_2)\cong \F_{d}$, $H_*(Y_2)\cong0$.
    \end{enumerate}

    Moreover case $(1)$ holds if the Maslov grading of the generator of minimal Alexander grading at least $2$ is odd, while condition $(2)$ holds if it is even.
\end{lemma}

\change{Here, and for the remainder of this paper, the subscript on an $\F$-summand indicates the Maslov index of that summand.}

\begin{proof}
Again this follows directly from work of Ozsv\'ath-Szab\'o in the setting in which $K$ is an $L$-space knot as opposed to an almost $L$-space knot~\cite[Section 3]{ozsvath2005knot} given Lemma~\ref{spin0}.
\end{proof}

It remains to determine $H_*(\{(0,j)\})$ for $|j|\leq 1$.

To do so we combining Lemmas~\ref{alexander1},~\ref{Xtom} and~\ref{LemmaX1} and obtain the following:

\begin{enumerate} 
\item Suppose $\rank(H_*(\change{\{}i<0,j=1\change{\}}))=0$. Then $\rank(H_*(\change{\{}i\leq 0,j=1\change{\}}))=\rank(H_*(\change{\{(}0,1\change{)\}}))$ is $1$ or $2$.

\item Suppose $\rank(H_*(\change{\{}i<0,j=1\change{\}}))=1$. Then:

\begin{enumerate}
      \item If $\rank(H_*(\change{\{}i\leq 0,j=1\change{\}}))=2$ then $\rank(H_*\change{(\{}(0,1)\change{\})})$ is either $1$ or $3$.
    \item If $\rank(H_*(\change{\{}i\leq 0,j=1\change{\}}))=1$ then $\rank(H_*\change{(\{}(0,1)\change{\})})$ is either $0$ or $2$.
 
\end{enumerate}

\end{enumerate}

We seek to exclude the case that $\rank(H_*\change{(\{}(0,1)\change{\})})=3$.

\begin{lemma}
Suppose $K$ is an almost $L$-space knot. Then $\rank(H_*(\change{\{(}0,1\change{)\}}))\neq3$
\end{lemma}

\begin{proof}

    If $\rank(H_*\change{(\{}(0,1)\change{\})})=3$ then  $$\rank(H_*(\change{\{}i\leq 0,j=1\change{\}}))=\rank(H_*\change{(}\{i<0,j=0\})\change{)}=2.$$ We have the following cases:\begin{enumerate}
        \item $\rank(H_*\change{(}\{i\leq 0,j=0\})\change{)}=1$ so that $\rank(H_*\change{(\{}(0,0)\change{\})})=1$ or $3$ by Lemma~\ref{Xtom}.
        \item $\rank(H_*\change{(}\{i\leq 0,j=0\})\change{)}=2$ so that $\rank(H_*\change{(\{}(0,0)\change{\})})=0$ $2$ or $4$ by Lemma~\ref{Xtom}, a contradiction since this rank should be odd.
    \end{enumerate}

    It thus remains to exclude the case that $$\rank(H_*(\change{\{}i\leq 0,j=1\change{\}}))=2, \rank(H_*\change{(\{}(0,1)\change{\})})=3$$ and $\rank(H_*\change{(}\{i\leq 0,j=0\})\change{)}=1$. This is straightforward as $$H_*(\change{\{}i\leq0,j=-1\change{\}})\cong H_*(\change{\{}i=-1,j\leq 0\change{\}})\cong H_*(\change{\{}i=0,j\leq 1\change{\}}),$$ which is of rank zero or $1$, as can be seen by applying the exact sequences in Figure~\ref{exacttriangles} to determine $H_*(Y_2)$. Applying the exact triangle in Figure~\ref{exacttriangle} we find that that $\rank(H_*(\change{\{(}0,-1\change{)\}}))$ is $0,1$ or $2$, a contradiction, since it is supposed to be $3$.\end{proof}

We also exclude the case that $H_*(\change{\{(}0,1\change{)\}})\cong 0$:

\begin{lemma}
    Suppose $K$ is an almost $L$-space knot. Then $H_*(\change{\{(}0,1\change{)\}})\not\cong 0$\change{.}
\end{lemma}

We note that if we counted all knots which admit a positive $L$-space surgery as almost $L$-space knots then this lemma would be false. In particular there exist $L$-space knots $K$ with $H_*(\change{\{(}0,1\change{)\}})\cong 0$. For any such $K$, $\rank(\widehat{\HF}(S^3_{2g(K)-2}(K)))=2g(K)$.

\begin{proof}
Suppose $H_*(\change{\{(}0,1\change{)\}})\cong 0$. Then we must have that $H_*(\change{\{}i\leq 0,j=1\change{\}})\cong H_*(\change{\{}i<0,j=0\change{\}})\cong\F$. We have two cases according to Lemma~\ref{alexander0}. Applying the exact triangle in Figure~\ref{exacttriangle} in each instance  we have that:\begin{enumerate}
    \item If $\rank(H_*(\change{\{}j=0,i\leq 0\change{\}}))=1$ then $\rank(H_*\change{(\{}(0,0)\change{\})})$ is $0$ or $2$, both of which are even, a contradiction.
    \item If $\rank(H_*(\change{\{}j=0,i\leq 0\change{\}}))=2$ then $\rank(H_*(\change{\{}j=-1,i<0\change{\}}))=2$. Since $\rank(H_*(\change{\{}j=-1,i\leq 0\change{\}}))=\rank(H_*(\change{\{}i=-1,j\leq 0\change{\}}))\leq 1$ as in the previous lemma, it follows that $H_*(\change{\{(}0,-1\change{)\}})$ is either $1$, $2$ or $3$, a contradiction.
\end{enumerate} \end{proof}

We are thus left with the cases:

\begin{enumerate} 
\item $\rank(H_*(\change{\{}i<0,j=1\change{\}}))=0$ and $\rank(H_*(\change{\{}i\leq 0,j=1\change{\}}))=\rank(H_*\change{(\{}(0,1)\change{\})})$ is $1$ or $2$

\item $\rank(H_*(\change{\{}i<0,j=1\change{\}}))=1$ and $\rank(H_*(\change{\{}i\leq 0,j=1\change{\}}))=2$ and $\rank(H_*\change{(\{}(0,1)\change{\})})=1$.

\item $\rank(H_*(\change{\{}i<0,j=1\change{\}}))=1$ and $\rank(H_*(\change{\{}i\leq 0,j=1\change{\}}))=1$ and $\rank(H_*\change{(\{}(0,1)\change{\})})=2$.

\end{enumerate}

We now compute $H_*\change{(\{}(0,0)\change{\})}$, again by cases according to whether $\rank(H_*(\change{\{}i\leq 0,j=1\change{\}}))=\rank(H_*(\change{\{}i<0,j=0\change{\}}))$ is $1$ or $2$:

\begin{enumerate}
             \item If $\rank(H_*(\change{\{}i<0,j=0\change{\}}))=2$ then $\rank(H_*\change{(\{}(0,0)\change{\})})$ is $0,1,2,3$ or $4$ by Lemma~\ref{Xtom}. The even cases are excluded as $\rank(H_*\change{(\{}(0,0)\change{\})})$ must be odd and we must have that $\rank(H_*(\change{\{}i\leq 0,j=0\change{\}}))=1$.
            \item\label{note:splits} If $\rank(H_*(\change{\{}i<0,j=0\change{\}}))=1$ then $\rank(H_*\change{(\{}(0,\change{0})\change{\})})$ is $0,1,2$ or $3$ by Lemma~\ref{Xtom}. The even cases are excluded as before and we must have that $\rank(H_*(\change{\{}i\leq 0,j=0\change{\}}))=2$. 
        \end{enumerate}

In sum we have the following possibilities; $\widehat{\HFK}(K,1)\cong \F$ or $\F^2$ while $\widehat{\HFK}(K,0)\cong \F$ or $\F^3$.

 \section{Maslov Gradings}\label{maslov}
We now proceed to compute the Maslov gradings of the bigraded vector spaces produced in the previous section. We proceed again by cases, following the argument given by Ozsv\'ath-Szab\'o for their corresponding step in the $L$-space knot case. Essentially this amounts to keeping track of the Maslov gradings in the exact triangles shown in Figure~\ref{exacttriangles} and Figure~\ref{exacttriangle}. Specifically we apply the fact that the maps $f,a,e,b$ in Figure~\ref{exacttriangles} preserve the Maslov grading, while maps $c$ and $d$ each lower it by $1$. Likewise we use the fact that in Figure~\ref{exacttriangle} the maps $a$ and $b$ preserve the Maslov grading while map $c$ decreases it by $1$.

We assume throughout this section that the genus of the almost $L$-space knot in question is strictly greater than one. We are safe in this assumption as genus $1$ almost $L$-space knots are rank at most two in their maximal Alexander grading and the only such knots are $T(2,-3)$, the mirror of $5_2$ and the figure eight knot by results of Ghiggini~\cite{ghiggini2008knot} and Baldwin-Sivek~\cite{baldwin2022floer} -- see Corollary~\ref{genus1class}.

\begin{lemma}\label{maslovgradingcomp}
Suppose $K$ is an almost $L$-space knot of genus at least two. Let $x$ be the generator of lowest Alexander grading strictly greater than $1$. Let $m$ be the Maslov grading of $x$ and $A$ be its Alexander grading. We have that either:

 $m$ is odd and one of the following holds:

\begin{enumerate}
    \item[\textbf{1ai)}] $ \widehat{\HFK}(K,1)\cong\F_{m+1}$ and $\widehat{\HFK}(K,0)\cong\F_{m}$\change{,}
    
    \item[\textbf{1aii)}] $\widehat{\HFK}(K,1)\cong\F_d$ and $\widehat{\HFK}(K,0)\cong\F_{m-1}\oplus\F_{d-1}\oplus\F_{d-1}$\change{,}

 \item[\textbf{1bi)}] $\widehat{\HFK}(K,1)\cong\F_{m-1}\oplus\F_{m-2}$ and $\widehat{\HFK}(K,0)\cong\F_{m-3}$\change{,}
    
    \item[\textbf{1bii)}] \change{There exists an integer $a$ such that} $\widehat{\HFK}(K,1)\cong\F_a\oplus\F_{m-1}$ and $\widehat{\HFK}(K,0)\cong\F_{a-1}\oplus\F_{m-2}\oplus\F_{a-1}$\change{;}

\end{enumerate}
 or $m$ is even and one of the following holds:
\begin{enumerate}

  \item[\textbf{2ai)}]
    
    $\widehat{\HFK}(K,1)\cong\F_{m+1-2A}$ and $\widehat{\HFK}(K,0)\cong\F_{m-2A}$\change{,}

    \item[\textbf{2aii)}] $\widehat{\HFK}(K,1)\cong\F_b$ and $\widehat{\HFK}(K,0)\cong \F_{m-2A+1}\oplus \F_{b-1}\oplus\F_{b-1}$\change{,}

\item[\textbf{2bi)}] $\widehat{\HFK}(K,1)\cong\F_{m-2A+1}\oplus\F_{m-2A}$, $\widehat{\HFK}(K,0)\cong\F_{m-2A+1}$\change{,}

    \item[\textbf{2bii)}] \change{There exists an integer $b$ such that }$\widehat{\HFK}(K,1)\cong\F_b\oplus\F_{m-2A+3}$, $\widehat{\HFK}(K,0)\cong\F_{b-1}\oplus\F_{b-1}\oplus \F_{m-2A+2}$ \change{.}
\end{enumerate}

\end{lemma}

\begin{proof}
We proceed by cases\change{. In each case, we apply the exact triangles~\ref{exacttriangles} and~\ref{exacttriangle} as outlined above. We give a more detailed description of the computation in case 1ai) and suppress some of the details in subsequent cases.}
\begin{enumerate}
  \item Suppose $m$ is odd and, per Lemma~\ref{oddeven}, $H_*(\change{\{}i<0,j=1\change{\}})\cong 0$. \change{The the Maslov grading preserving map $b: H_*(\{i\leq 0,j=1\})\to H_*(\{(0,1)\})$ from exact triangle~\ref{exacttriangle} with $m=1$ is an isomorphism.}\begin{enumerate}
        \item Suppose $H_*\change{(\{}(0,1)\change{\})}\cong\F_d$. It follows that $H_*(\change{\{}i\leq 0,j=1\change{\}})\cong\F_d$\change{.} \change{Using the $(i, j)$ symmetry of $\CFK^\infty$, we have that $$H_*(\change{\{}i\leq 0,j=1\change{\}})\cong H_*(\change{\{}j\leq 0,i=1\change{\}})\cong\F_d,$$ so that $H_*(\change{\{}j\leq -1,i=0\change{\}})\cong\F_{d-2}$, and indeed $H_*(\change{\{}i<0,j=0\change{\}})\cong\F_{d-2}$, again by the $(i,j)$-symmetry of $\CFK^\infty(K)$. In particular we are in case~\ref{note:splits} from the final analysis in Section~\ref{nomaslov}. There are two subcases, according to whether $\rank(H_*(\{(0,0)\}))$ is one or three.}

        \begin{enumerate}
        \item Suppose $H_*\change{(\{}(0,0)\change{\})}\cong\F_a$. \change{As noted in case~\ref{note:splits} from the final analysis in Section~\ref{nomaslov},} $H_*(\change{\{}i\leq 0,j=0\change{\}})\cong\F_b\oplus\F_c$ \change{for some pair of integers $b,c$. Thus the exact sequence~\ref{exacttriangle} with $m=0$ splits with the injective map $H_*(\{i<0,j=0\})\inj H_*(\{i\leq 0,j=0\})$ and the surjective map $H_*(\{i\leq 0,j=0\})\surj H_*(\{(0,0)\})$ both preserving the Maslov grading. Thus, without loss of generality, we have the following equalities of integers;} $a=b$ and $d-2=c$. It follows that $H_*(\change{\{}i<0,j=-1\change{\}})\cong \F_{a-2}\oplus\F_{d-4}$. \change{Now, from the fact that $H_*(\{(0,1)\})\cong \F_{d}$, we can deduce from the symmetry properties of $\CFK^\infty(-)$ that $H_*(\{(0,-1)\})\cong \F_{d-2}$. Moreover, again by symmetry properties of $\CFK^\infty(-)$, we have that $$H_*(\{j=-1,i\leq 0\})\cong H_*(\{i=-1,j\leq 0\})\cong H_*(Y_2),$$ which is rank one per Lemma~\ref{oddeven} and the assumption that $m$ is odd. Thus exact triangle~\ref{exacttriangle} splits for $m=-1$. Since the map $$H_*(\{(0,-1)\})\to H_*(\change{\{}i<0,j=-1\change{\}})\cong \F_{a-2}\oplus\F_{d-4}$$ lowers the Maslov index by one, it} follows in turn that $d-3=a-2$. \change{ It likewise follows that} $H_*(\change{\{}i\leq 1,j=-1\change{\}})\cong \F_{d-4}$. \change{Now note that $H_*(\{(0,n)\})\cong H_*(\{(0,-n)\})\cong 0$ for $-1>n>-A$ and $H_*(\{(0,-A)\})\cong\F_{m-2A}$ by the symmetry properties of $\CFK^\infty(-)$. Thus by repeated application of the exact triangle~\ref{exacttriangle} decreasing $m$ from $-1$ to $-A$, we can conclude that} $d-4-2(A-1)=m-2A-1$, so that $d=m+1$.

        \item  Suppose $H_*\change{(\{}(0,0)\change{\})}\cong\F_a\oplus\F_b\oplus \F_c$ \change{ for some triple of integers $a,b,c$}. Then without loss of generality $d-2=c-1$  and $H_*(\change{\{}i\leq 0,j=0\change{\}})\cong \F_a\oplus \F_b$, so that $H_*(\change{\{}i< 0,j=-1\change{\}})\cong \F_{a-2}\oplus \F_{b-2}$. It follows without loss of generality that $d-3=b-2$. Indeed, $H_*(\change{\{}i\leq 1,j=-1\change{\}})\cong \F_{a-2}$. We then find that $a-2-2(A-1)=m-2A-1$, so that $a=m-1$.

        \end{enumerate}

        \item Suppose $H_*\change{(\{}(0,1)\change{\})}\cong\F_a\oplus\F_b$\change{, for some pair of integers $a,b$}. Then $H_*(\change{\{}i\leq 0,j=1\change{\}})\cong\F_a\oplus\F_b$.
        
        \begin{enumerate}
            \item If $H_*\change{(\{}(0,0)\change{\})}\cong \F_c$ \change{for some integer $c$,} then $H_*(\change{\{}i\leq 0,j=0\change{\}})\cong\F_{b-2}$, $c-1=a-2$. It follows that $H_*(\change{\{}i<0,j=-1\change{\}})\cong\F_{b-4}$. Thus $a-3=b-4$. As before we find that $a=m-1$.

            \item If $H_*\change{(\{}(0,0)\change{\})}\cong\F_c\oplus\F_d\oplus\F_e$ \change{for some triple of integers $c,d,e$, }then without loss of generality $e-1=a-2, d-1=b-2$ and $H_*(\change{\{}i\leq 0,j=0\change{\}})\cong\F_{c}$. It follows that $c-2=b-3$ or $a-3$. Suppose $c-2=b-3$. Then we find that $a-2-2(A-1)=m-2A-1$ so that $a=m-1$. Suppose $c-2=a-3$ then we find that $b=m-1$.

        \end{enumerate}
    \end{enumerate}
   \item Suppose $m$ is even and, per Lemma~\ref{oddeven} $H_*(\change{\{}i<0,j=1\change{\}})\cong \F_a$ \change{for some integer $a$}. Then $a=m-2(A-1)$. Moreover:
       \begin{enumerate}
        \item Suppose $H_*\change{(\{}(0,1)\change{\})}\cong\F_b$ \change{for some integer $b$}. Then $H_*(\change{\{}i\leq 0,j=1\change{\}})\cong\F_a\oplus\F_b$ and $H_*(\change{\{}i<0,j=0\change{\}})\cong\F_{a-2}\oplus\F_{b-2}$
        \begin{enumerate}
            \item Suppose $H_*\change{(\{}(0,0)\change{\})}\cong\F_c$ \change{for some integer $c$}. Suppose  $a-2=c-1$ and $H_*(\change{\{}i\leq 0,j=0\change{\}})\cong\F_{b-2}$. It follow that $H_*(\change{\{}i< 0,j=-1\change{\}})\cong\F_{b-4}$, a contradiction. Thus $b-2=c-1$, and $H_*(\change{\{}i\leq 0,j=0\change{\}})\cong\F_{a-2}$. It follows that $b-3=a-4$.

            \item Suppose $H_*\change{(\{}(0,0)\change{\})}\cong F_c\oplus\F_d\oplus\F_e$ \change{for some triple of integers $c,d,e$}. Then without loss of generality $c-1=a-2,d-1=b-2$ and $H_*(\change{\{}i\leq 0,j=0\change{\}})\cong\F_e$. It follows that $e-2$ is $b-3$.
        \end{enumerate}

        \item Suppose $H_*\change{(\{}(0,1)\change{\})}\cong\F_b\oplus\F_c$ \change{for some pair of integers $b,c$}. It follows without loss of generality that $c-1=a$ and $H_*(\change{\{}i\leq 0,j=1\change{\}})\cong\F_b$ while $\F_{b-2}\cong H_*(\change{\{}i<0,j=0\change{\}})$.

        \begin{enumerate}
            \item Suppose $H_*\change{(\{}(0,0)\change{\})}\cong\F_d$ \change{for some integer $d$}. Then $H_*(\change{\{}i\leq 0,j=0\change{\}})\cong\F_{b-2}\oplus\F_d$ and $H_*(\change{\{}i< 0,j=0\change{\}})\cong\F_{b-4}\oplus\F_{d-2}$. It follows that $b-3=d-2$, $a-3=d-2$.

            \item Suppose $H_*\change{(\{}(0,0)\change{\})}\cong\F_d\oplus\F_e\oplus\F_f$ \change{for some triple of integers $d,e,f$}. It follows that $f-1=b-2$ without loss of generality. Similarly $H_*(\change{\{}i\leq 0,j=0\change{\}})\cong\F_d\oplus\F_e$. Thus, $H_*(\change{\{}i<0,j=-1\change{\}})\cong\F_{d-2}\oplus\F_{e-2}$. Thus $\{d-2,e-2\}=\{b-3,c-3\}$ i.e. without loss of generality $d=b-1,e=c-1$.
        
        \end{enumerate}
    \end{enumerate}
       
\end{enumerate}\end{proof}

\section{The \change{filtered} chain homotopy type of $\CFK^\infty$}\label{chain}

In this section we seek to determine the filtered chain homotopy type of $\CFK^\infty$ of almost $L$-space knots. We have $8$ cases to deal with according to Lemma~\ref{maslovgradingcomp}, although we will see that there is a certain amount of degeneracy amongst these cases.

We again assume throughout this section that the genus of the almost $L$-space knot in question is strictly greater than one\change{, as as we did Section~\ref{maslov}}.

\begin{lemma}\label{grading11} Let $K$ be an almost $L$-space knot of genus strictly greater than $1$. Then \change{there is an even integer $d$ such that one of the two following statements hold}:\begin{enumerate}
    \item\label{p1} $H_*(\change{\{}i=0,j>1\change{\}})\cong \F_0$,  $H_*(\change{\{}i=0,j<-1\change{\}})\cong\F_{\change{d-4}} $ and $H_*(\change{\{}i=0,|j|\leq 1\change{\}})\cong\F_{\change{d-3}}$.
    \item\label{p2}    $H_*(\change{\{}i=0,j>1\change{\}})\cong \F_0\oplus \F_{\change{d+1}}$, $H_*(\change{\{}i=0,j<-1\change{\}})\cong 0$ and $H_*(\change{\{}i=0,|j|\leq 1\change{\}})\cong\F_{\change{d}}$.
\end{enumerate}

\end{lemma}

\begin{proof}
\change{Each case in this Lemma will correspond to one of the two cases in Lemma~\ref{oddeven} namely the case that $(H_*(X_2),H_*(Y_2))$ is $(\F_d,0)$  or that it is $(0,\F_d)$ where here $d$ is an even integer. First note that by the symmetry properties of $\CFK^\infty(K)$}, $$H_*(\change{\{}i=0,j<-1\change{\}})\cong H_*(\change{\{}j=0,i<-1\change{\}})\cong  H_*(\change{\{}j=2,i\leq 0\change{\}})\change{[-4]}\cong H_*(X_2)\change{[-4]},$$ \change{where here $[-]$ indicates a shift in the Maslov grading}. Moreover, $H_*(Y_2)=H_*(\change{\{}i=0,j\leq 1\change{\}})$, \change{proving the second parts of~\ref{p1} and~\ref{p2} in the statement of the Lemma}.

$H_*(\change{\{}i=0,j>1\change{\}})$ can be computed using the is a short exact sequence:
\begin{center}
\begin{tikzcd}
    0\arrow[r]&C(\change{\{}i=0,j\leq 1\change{\}})\arrow[r]&C(\change{\{}i=0\change{\}})\arrow[r]&C(\change{\{}i=0,j>1\change{\}})\arrow[r]&0
\end{tikzcd}

\end{center}
which yields exact triangle on homology:
\begin{center}
\begin{tikzcd}
    H_*\change{(}\{i=0,j\leq 1\}\change{)}\arrow[rr]&&H_*(\change{\{}i=0\change{\}})\cong\F_0\arrow[dl]\\&H_*(\change{\{}i=0,j> 1\change{\}})\arrow[ul,]
\end{tikzcd}
\end{center}

It follows that $H_*(\change{\{}i=0,j>1\change{\}})$ is $\F_0$ or $\F_0\oplus\F_{d\change{+1}}$ according to whether or not $H_*\change{(}\{i=0,j\leq 1\}\change{)}$ is $\F_{\change{d}}$ or $0$. Note here that Ozsv\'ath-Szab\'o show that $\change{d}<0$~\cite[Proof of Theorem 1.2]{ozsvath2005knot}, so the top map is necessarily trivial, \change{proving the first parts of~\ref{p1} and~\ref{p2} in the statement of the Lemma}.

There is another short exact sequence:
\begin{center}
\begin{tikzcd}
    0\arrow[r]&C\change{(}\{i=0,j< -1\}\change{)}\arrow[r]&C\change{(}\{i=0,j\leq 1\}\change{)}\arrow[r]&C\change{(}\{i=0,|j|\leq 1\}\change{)}\arrow[r]&0
\end{tikzcd}

\end{center}
giving an exact triangle on homology:
\begin{center}
\begin{tikzcd}
    H_*\change{(}\{i=0,j< -1\}\change{)}\arrow[rr,]&&H_*\change{(}\{i=0,j\leq 1\}\change{)}\arrow[dl]\\&H_*(\change{\{}i=0,|j|\leq 1\change{\}})\arrow[ul]
\end{tikzcd}
\end{center}
$H_*(\change{\{}i=0,|j|\leq 1\change{\}})\cong \F_{\change{d-3}}$ or $\F_{\change{d}}$ in the case that that $(H_*(X_2),H_*(Y_2))$ is either \change{ $(0,\F_{\change{d}})$ or $(\F_{\change{d}},0)$}, \change{proving the final parts of~\ref{p1} and~\ref{p2} in the statement of the Lemma}.
\end{proof}

For grading reasons this determines the  filtered chain homotopy type of the chain complex $C(\change{\{}(x,y):|y-x|>1\change{\}})$, just as in the case for $L$-space knots.

We seek to determine the rest of the chain complex.

\begin{lemma}
    Suppose $K$ is an almost $L$-space knot and $\rank(H_*(\change{\{}(0,0)\change{\})})=1$. Then $\CFK^\infty(K)$ is an almost staircase complex.
\end{lemma}

\begin{proof}
We know that $H_*(\change{\{(}i=0\change)\})\cong\F_0$ and that $\rank(H_*(\change{\{(}0,\change{n})\change{\}))}\leq 1$ for $\change{n}\neq \pm1$, $1\leq \rank(H_*(\change{\{(}0,1)\change{\}))}\leq 2$. Moreover, in $\widehat{\HFK}(K)$ -- which is identified with $\change{\underset{n\in\Z}{\bigoplus}}H_*(\change{\{}(n,0)\change{\}})$ -- the \change{$(A,M)$}-gradings of the $5$ generators of smallest $A$-grading \change{(in absolute value)} are either $(0,d),(1,d+1),(-1,d-1),(A,d),(-A,d-2A)$ if $d$ is odd or $(0,d),(1,d+1),(-1,d-1),(-A,d),(A,d+2A)$ if $d$ is even. We have a number of cases:

\begin{enumerate}
    \item If $d$ is even then:
    
    \begin{enumerate}
        \item if $A>1$ then $H_*(\change{\{}i=0,j\leq 1)\cong 0$ and $H_*(\change{\{}i=0,j\geq A\change{\}})\cong \F_0$. This forces $C(\change{\{}i=0\change{\}})$ to be of the desired form -- i.e. as determine by the $\CFK^\infty$-type of an almost staircase complex of type $2$ -- perhaps with the addition of components of the differential from $(0,0)$ to $(0,-A-1)$ -- if there exists a generator of this grading -- and from $(0,-1)$ to $(0,-A')$, where $A'$ is the smallest integer $A'>A+1$ for which $C\change{(\{}(0,-A')\change{\})}$ is non-trivial. $C(\change{\{}j=0\change{\}})$ is determined similarly, up to the addition of two components of the differential. In $\CFK^\infty$ the resulting additional 4 components of the differential can be removed by a filtered chain homotopy. In order that $\partial^2=0$ there must be two diagonal components to the differential on $\CFK^\infty(K)$, as shown in Figure~\ref{almoststaircase1fig}.

        \item if $A=1$ then the fact that $H_*(\change{\{}i=0\change{\}})\cong \F_0$, $H_*(\change{\{}i=0,j\leq 1\change{\}})\cong \F_{d+2}, H_*(\change{\{}i=0,j\leq -1\change{\}})\cong \F_{d}, H_*(\change{\{}i=0,j< -1\change{\}})\cong 0$ determines $C(\change{\{}i=0\change{\}})$, perhaps with the addition of components of the differential from $(0,0)$ to $(0,-2)$ -- if there exists a generator of this grading -- and from the generator of Maslov grading $d-1$ in bigrading $(0,-1)$ to $(0,-A')$, where $A'$ is the smallest integer $A'>2$ for which $C(\change{\{(}0,-A\change{)\}})$ is non-trivial. $C(\change{\{}j=0\change{\}})$ is determined similarly, up to the addition of two components of the differential. In $\CFK^\infty$ the resulting additional 4 components of the differential can be removed by a filtered chain homotopy. In order that $\partial^2=0$ there must be two diagonal components to the differential on $\CFK^\infty(K)$, as shown in Figure~\ref{almoststaircase1fig}.
    \end{enumerate}

\item If $d$ is odd then:

\begin{enumerate}
    \item if $A>1$ then ${H_*(\change{\{}i=0,j\leq 1\change{\}})\cong \F_{d-1}}$ and ${H_*(\change{\{}i=0,j>1\change{\}})\cong \F_0\oplus\F_d}$. This once again determines $C(\change{\{}i=0\change{\}})$, perhaps up to the addition of two components of the differential one from a generator of bigrading $(0,A+1)$ -- if such a generator exists -- to the generator of grading $(0,0)$ and from $(0,A')$ to $(0,1)$ where $A'$ is the smallest $A'>A+1$ such that $C\change{(\{}(0,A')\change{\})}$ is non-trivial. $C(\change{\{}j=0\change{\}})$ is determined similarly, up to the addition of two components of the differential. In $\CFK^\infty$ the resulting additional 4 components of the differential can be removed by a filtered chain homotopy. In order that $\partial^2=0$ there must be two diagonal components to the differential on $\CFK^\infty(K)$, as shown in Figure~\ref{almoststaircase2fig}.
    
    \item if $A=1$ then the fact that $H_*(\change{\{}i=0\change{\}})\cong H_*(\change{\{}i=0,j>1\change{\}})\cong \F_0$, $H_*(\change{\{}i=0,j\leq 1\change{\}})\cong 0$, $H_*(\change{\{}i=0,j\leq 1\change{\}})\cong 0$ determines $C(\change{\{}i=0\change{\}})$ as being of the desired form perhaps with addition of two unwanted components. One of these components is from the generator of bigrading $(0,2)$ -- if it exists --  to the generator of grading $(0,0)$, the other is a component from from $(0,A')$ to the generator in bigrading $(0,1)$ of Maslov grading $d+1$ where $A'$ is the smallest $A'>2$ such that $C(\change{\{(}0,A'\change{)\}})$ is non-trivial. $C(\change{\{}j=0\change{\}})$ is similarly determined.  In $\CFK^\infty$ the resulting additional 4 components of the differential can be removed by a filtered chain homotopy. In order that $\partial^2=0$ there must be two diagonal components to the differential on $\CFK^\infty(K)$, as shown in Figure~\ref{almoststaircase2fig}.
\end{enumerate}

\end{enumerate}
\end{proof}

To deal with the case that $\rank(H_*(\change{\{(}0,0\change{)\}}))=3$ we first understand the behavior of the complex near the diagonal:

\begin{lemma}
    Suppose $K$ is an almost $L$-space knot and $\rank(H_*(\change{\{(}0,0\change{)\}}))=3$. Then up to a filtered change of basis $C_*(\{|j-i|\leq 1\})$ is the direct sum of a box complex and a staircase complex.
\end{lemma}

\begin{proof}

We proceed by cases according to $\rank(H_*\change{(\{}(0,1)\change{\})})$.

Suppose $H_*\change{(\{}(0,1)\change{\})}\cong\F_a, H_*\change{(\{}(0,0)\change{\})}\cong\F_{a-1}^2\oplus\F_d, H_*\change{(\{}(0,-1)\change{\})}\cong\F_{a-2}$. Note that these Maslov gradings are determined by Lemma~\ref{maslovgradingcomp}.  Let $x,y_1,y_2,z,w$ be the respective generators.  Note that there is a unique \change{form} $\partial$ can take on $C(\change{\{}i=0\change{,}|j|\leq 1\change{\}})$ such that $H_*(\change{\{}i=0,|j|\leq 1\change{\}})\cong\F_d$, up to a basis change in $C\change{(\{}(0,0)\change{\})}$. It follows from here that $(C(\change{\{(}i,j\change{)}:|j-i|\leq 1\change{\}}),\partial)$ is determined up to the addition of diagonal components of $\partial$. The only way to add diagonal components is to have diagonal components from $w$ to $U^{n+1}x,U^{n}w$ for some $n$ or vice versa. A filtered change of basis removes these components.

Suppose now that $H_*\change{(\{}(0,1)\change{\})}\cong\F_m\oplus \F_d, H_*\change{(\{}(0,0)\change{\})}\cong\F_{m-1}^2\oplus\F_{d-1}$. Then we have that $ H_*\change{(\{}(0,-1)\change{\})}\cong\F_{m-2}\oplus\F_{d-2}$. Note that these Maslov gradings are determined by Lemma~\ref{maslovgradingcomp}. Let $x_1$ be the generator of $H_*\change{(\{}(0,1)\change{\})}$ of Maslov grading $m$, $y_1$ be the generator of $H_*\change{(\{}(0,1)\change{\})}$ of Maslov grading $d$, $x_2,x_3$ be the generators of $H_*\change{(\{}(0,0)\change{\})}$ of Maslov grading $m-1$, $y_2$ be the generator of $H_*\change{(\{}(0,0)\change{\})}$ of Maslov grading $d-1$, $x_4$ be the generator of $H_*\change{(\{}(0,-1)\change{\})}$ of Maslov grading $m-2$ and $y_3$ be the generator of $H_*\change{(\{}(0,-1)\change{\})}$ of Maslov grading $d-2$.

Suppose $m-1\not\in\{d,d-1,d-2\}$. Consider the restriction of the differential to $C(\change{\{}i=0,|j|\leq 1\change{\}})$.  Then after a basis change we may take $\partial x_1=x_2$. Since $m-1\neq d,d-1$, we have $\partial x_3=x_4$. The remaining component of the differential is then determined. Specifically, after mirroring we may assume that $y_1$ is the generator that persists and the remaining component of the differential is $\partial y_2=y_3$. 

We then have that in $C(\change{\{}|j-i|\leq 0\change{\}})$ the $y$ generators form a single staircase and the $x$ generators form boxes. The question remains over whether or not there are diagonal components. By inspection these come in pairs and can be removed by a filtered change in basis.

We now deal individually with the cases $m=d+1$, $m=d$ and $m=d-1$:

\begin{enumerate}
    \item $m=d+1$ \begin{enumerate}\item Suppose the Maslov grading $d-2$ generator is not the one which persists to $H_*(\change{\{}i=0,|j|\leq 1\change{\}})$. Then there is a component of $\partial$ from the Maslov index $d-1$ generator to the Maslov index $d-2$ generator. We have the following subcases:\begin{enumerate}
        \item Suppose the generator of $H_*\change{(\{}(0,1)\change{\})}$ of Maslov index $d$ generator does not persist to $H_*(\change{\{}i=0,|j|\leq 1\change{\}})$. Then it has a component to the generator of $H_*\change{(\{}(0,-1)\change{\})}$ of  Maslov index $d-1$. The Maslov index $d+1$ generator must have a component to one of the generators of $H_*\change{(\{}(0,0)\change{\})}$ of Maslov index $d$. Such a $\partial$ clearly cannot be extended to $C(\change{\{}|j-i|\leq 1\change{\}})$ a contradiction.

        \item Suppose the generator of $H_*\change{(\{}(0,1)\change{\})}$ of Maslov index $d$ persists. After a change of basis this determines the vertical components of the differential and it is readily observed that $C(\change{\{}|j-i|\leq 1\change{\}})$ is the direct sum of a staircase and a box complex.
    \end{enumerate}
    \item Suppose the Maslov index $d-2$ generator does persist to ${H_*(\change{\{}i=0,|j|\leq 1\change{\}})}$. This is impossible because it is of the wrong Maslov grading.
\end{enumerate}
    \item $m=d$. In this case the chain complex is thin and so splits as a direct sum of boxes and staircases by work of Petkova~\cite[Lemma 7]{petkova2013cables}.

    \item $m=d-1$\begin{enumerate}
        \item Suppose the Maslov index $d-3$ generator is not the one which persists to $H_*(\change{\{}i=0,|j|\leq 1\change{\}})$. After a change of basis we may assume the differential has a component from a single Maslov index $d-2$ generator to it. We have the following cases:\begin{enumerate}
            \item Suppose the other generator of $H_*\change{(\{}(0,0)\change{\})}$ of Maslov grading $d-2$ does not persist to $H_*(\change{\{}i=0,|j|\leq1\change{\}})$. Then there is a component 
of the differential from the the generator of $H_*\change{(\{}(0,1)\change{\})}$ of Maslov index $d-1$ to it and indeed the whole vertical complex is determined. Indeed, the whole complex is seen to be a staircase complex plus a box complex. There can be no diagonal components to the differential for for grading reasons.

            \item Suppose the generator of $H_*\change{(\{}(0,0)\change{\})}$ of Maslov index $d-2$ persists to $H_*(\change{\{}i=0,|j|\leq 1\change{\}})$. Then the vertical components of $\partial$ are determined by the Maslov gradings. It is readily observed that this does not extend to a differential on $C(\change{\{}|j-i|\leq 1\change{\}})$.
        \end{enumerate}
    \end{enumerate} 

\end{enumerate}\end{proof}

This determines the complex up to addition of additional arrows between $C(\change{\{(}i,j\change{)}:|j-i|\leq 1\change{\}}), C(\change{\{(}i,j\change{)}:j-i>1\change{\}}), C(\change{\{(}i,j\change{)}:i-j>1\change{\}})$. In order that $H_*(\change{\{}i=0)\cong H_*(\change{\{}j=0\change{\}})\cong\F_0$ it is readily seen that the staircases from $C(\change{\{}|j-i|>1\change{\}})$ and $C(\change{\{}|j-i|<1\change{\}})$ connect to form a large staircase. We now show that up to filtered chain homotopy there are no diagonal components:

\begin{proposition}\label{stair}
Suppose $K$ is an almost $L$-space knot and $\rank(H_*\change{(\{}(0,0)\change{\})})=3$. Then $\CFK^\infty(K)$ is the direct sum of a box complex and a staircase complex.
\end{proposition}

Before proving this result we find it convenient to broaden our definition of staircase complexes to include their duals. We proceed with this in mind.

\begin{proof}

We have shown that $\CFK^\infty(K)$ contains a box complex and a staircase complex. Call them $B$ and $C$ respectively. We need to show that there is no component of $\partial$ from $C$ to $B$ after a filtered change of basis. The fact that there is no component from $C$ to $B$ then follows by dualizing the complex and the fact we have broaden our definition of staircase complexes.

Suppose towards a contradiction that there is some component of $\partial$ from $C$ to $B$. Consider the generators of lowest $j$-grading. Amongst these consider that of the highest $i$ grading. There is a unique such generator, call it $z$.

 Let $x_1, x_2,x_3,x_4$ be the generators of $C$ ordered as in Figure~\ref{boxfig}. Consider the components of $\partial z$ in $C$ which has the lowest $j$ grading. Amongst these consider those with the lowest $i$ grading.

Suppose this generator is $x_4$. In order that $\partial^2=0$, we must have that $\partial z$ has a component $z'$ of the same $j$ grading such that $\partial z'$ has a component $x_3$, or there is a component of the differential to $x_1$. In the first case either a filtered change of basis $z'\mapsto z'+x_4$ or $z\mapsto z+x_2$ removes both unwanted components of the differential. In the second case the change of basis $z\mapsto z+x_2$ removes the unwanted components of the differential.

We proceed now to the case that the generator is $x_3$. Performing the filtered change of basis $z\mapsto z+x_1$ or $z\mapsto z+x_4$ removes the unwanted components of the differential.

Suppose this generator is $x_2$. In order that $\partial^2=0$ we have that there is a generator $z'$ with the same $j$ grading as $z$ such that $\langle \partial z,z'\rangle\neq 0$ and that there is a component of the differential from $z'$ to $x_4$. In fact, again in order that $\partial^2=0$, we must also have a component of $\partial$ from $z'$ to $x_1$. Performing a filtered change of basis $z'\mapsto z'+x_2$ removes the unwanted components of the differential.

Suppose this generator is $x_1$. We must have a generator $z'$ of the same $j$ grading as $z$ such that there is a component of the differential from $z$ to $z'$ and there is a component of $\partial$ from $z'$ to $x_3$. The filtered change of basis $z'\mapsto z'+x_1$ then removes the unwanted components of the differential.\end{proof}

\section{Applications}\label{applications}

In this section we prove the applications advertised in the introduction.

\begin{corollary}[\cite{baldwin2022characterizing}]\label{genus1class}
Suppose $K$ is a genus one almost $L$-space knot. Then $K$ is $T(2,-3)$, the figure eight knot, or the mirror of the knot $5_2$.
\end{corollary}

\begin{proof}
\change{Suppose that} $K$ is an almost $L$-space knot \change{of genus one. Then $\widehat{\HFK}(K)$ is supported in Alexander gradings $A$ with $|A|\leq 1$. As noted at the conclusion of Section~\ref{nomaslov}, we have that} the rank of $\widehat{\HFK}(K)$ in Alexander grading $1$ is at most two. The result then follows immediately from Baldwin-Sivek's classification of nearly fibered knots~\cite{baldwin2022floer} -- i.e. genus one knots with knot Floer homology of maximal Alexander grading of rank two -- and Ghiggini's classification of genus $1$ knots with knot Floer homology of rank one in their maximal Alexander grading~\cite{ghiggini2008knot}.
\end{proof}

\begin{corollary}\label{comp}

The only composite almost $L$-space knot is $T(2,3)\# T(2,3)$.
\end{corollary}

\begin{proof}

Let $K$ be an almost $L$-space knot. Knots of genus at most $1$ are prime. If $g(K)>2$ then \change{$\change{\rank(\widehat{\HFK}(K,g(K)))}=1$ and $\rank(\widehat{\HFK}(K,g(K)-1)\leq 1$ by Theorem~\ref{infinityclassification}. Thus $K$ is fibered by work of Ghiggini~\cite{ghiggini2008knot} and Ni~\cite{ni2007knot}. It follows that $\rank(\widehat{\HFK}(K,g(K)-1))=1$}  by a result of Baldwin-Vela-Vick~\cite[Theorem 1.1]{baldwin_note_2018}. It follows that $K$ cannot be \change{composite} just as it does in the $L$-space case~\cite[\change{Corollary} 1.4]{baldwin_note_2018}.

Suppose $g(K)=2$ and $K=K_1\# K_2$. Then since $\widehat{\HFK}(K)\cong\widehat{\HFK}(K_1)\otimes\widehat{\HFK}(K_2)$ and $\widehat{\HFK}(K)$ is trivial in Alexander gradings $\change{\geq}2$ and of Maslov grading $0$ in Alexander grading $2$, it follows that $K_1, K_2$ are both genus $1$~\cite{ozsvath2005knot}, fibered~\cite{ghiggini2008knot} and strongly quasi-positive~\cite{hedden2010notions}. It follows that $K_1=K_2=T(2,3)$. It is readily checked that $T(2,3)\# T(2,3)$ is indeed an almost $L$-space knot.
\end{proof}

\begin{corollary}[\cite{baldwin2022characterizing}]~\label{fibered}
The mirror of $5_2$ is the only almost $L$-space knot which is not fibered.
\end{corollary}

\begin{proof}
If $K$ is an almost $L$-space knot with $g(K)>1$ then the rank of the knot Floer homology of $K$ in the maximal Alexander grading is $1$ by Theorem~\ref{infinityclassification}. It follows that $K$ is fibered by work of Ghiggini~\cite{ghiggini2008knot} and Ni~\cite{ni2007knot}\change{.}
\end{proof}

\begin{corollary}\label{tauk}
Suppose $K$ is an almost $L$-space knot for which $|\tau(K)|< g_3(K)$. Then $K$ is the figure eight knot.
\end{corollary}

\begin{proof}
    This follows immediately from Theorem~\ref{infinityclassification} and Corollary~\ref{genus1class}.
\end{proof}

\begin{corollary}[\cite{baldwin2022characterizing}]
    The only almost $L$-space knots that are not strongly quasi-positive are the figure eight knot and the left handed trefoil.
\end{corollary}

\begin{proof}
    If $K$ is an almost $L$-space knot of genus strictly greater than $1$, the result follows immediately from Theorem~\ref{infinityclassification}, Corollary~\ref{genus1class} and Hedden's result that knot Floer homology detects if fibered knots are strongly quasi-positive~\cite{hedden2010notions}. For the genus $1$-case, the result follows from Corollary~\ref{genus1class} together with the fact that the mirror of $5_2$ is strongly quasi-positive while the figure eight knot and $T(2,-3)$ are not.
\end{proof}

We now recover a result of Hedden:

\begin{corollary}[\cite{hedden2007knot,hedden2011floer}]~\label{Hedd}
    Suppose $K$ is a knot and $\rank(\widehat{\HFK}(K_n))=n+2$. Then $K$ is a non-trivial $L$-space knot and $n=2g(K)-1$.
\end{corollary}

Recall here that $K_n$ denotes the core of $n$-\change{surgery} on $K$. Again\change{,} we find it helpful to think of surgery in terms of Hanselman-Rasmussen-Watson's immersed curve invariants\change{, which allow us to readily compute the knot Floer homology of $K_n$~\cite[Section 4.3]{hanselman2018heegaard}. In slightly more detail, recall that if $\gamma_K$ is the immersed curve for $K$, viewed as living in the torus punctured at a basepoint $z$ then --- generally speaking --- the knot Floer homology of the core of $n$ surgery can be computed as the geometric intersection number of $\gamma_K$ with a curve of slope $n$  --- with respect to a specific framing of the torus --- starting and ending at $z$; see~\cite[Remark 53]{hanselman2018heegaard}.}

\begin{proof}

The immersed curve invariants of almost $L$-space knots \change{ are determined by Theorem~\ref{infinityclassification} via an application of~\cite[Proposition 48]{hanselman2018heegaard}, since each filtered basis in the statement of Theorem~\ref{infinityclassification} is both horizontally and vertically simplified (see Definition~\cite[Definition 11.23]{LOTborderedbook}.)}. It \change{then }follows \change{from the immersed curve formula for $\widehat{\HFK}(K_n)$ as discussed above} that if $K$ is an almost $L$-space knot then ${\rank(\widehat{\HFK}(K_n))>n+2}$. \change{More specifically, the immersed curve invariant --- viewed knot as a curve in the infinite strip $[-1,1]\times \R$ with punctures $\{(0,x):x\in\Z\frac{1}{2}\}$ --- of an almost $L$-space knot, when put in peg-board position, has a local max just above the ``peg" at $(0,\frac{1}{2})$. This implies that $\rank(\widehat{\HFK}(K_n))>\rank(\widehat{\HF}(S^3_n(K))=n+2$.} Indeed, we have that if $K$ is neither an $L$-space knot nor an almost $L$-space knot then $\rank (\widehat{\HFK}(K_n))\geq \rank(\widehat{\HF}(S^3_n(K)))>n+2$. Likewise it follows immediately from the classification of immersed curves for $L$-space knots \change{and the immersed curve formula for $\widehat{\HFK}(K_n)$ as discussed above} that if $K$ is an $L$-space knot then $\rank(\widehat{\HFK}(K_n))=n+2$ if and only if $K$ is non-trivial and $n=2g(K)-1$. In sum we have that $\rank(\widehat{\HFK}(K_n))=n+2$ if and only if $K$ is an $L$-space knot and $n=2g(K)-1$.
\end{proof}

\begin{proposition}\label{HFKcables}
      Suppose $K$ is an $L$-space knot. If a $2$ component link $L$ satisfies $\widehat{\HFK}(L)\cong\widehat{\HFK}(K_{2,4g(K)-2})$. Then $L$ is a $(2,4g(K)-2)$-cable of an $L$-space knot $K'$ such that $\widehat{\HFK}(K')\cong\widehat{\HFK}(K)$.
\end{proposition}

\begin{proof}[Proof of Proposition~\ref{HFKcables}]
Applying~\ref{Hedd}, this result follows as in the proof of~\cite[Theorem 4.1]{binns2022cable}.
\end{proof}

\begin{proposition}\label{HFLcables}
    Suppose $K$ is an $L$-space knot $\widehat{\HFL}(L)\cong\widehat{\HFL}(K_{m,2mg(K)-m})$. Then $L$ is a $(2,2mg(K)-m)$ cable of an $L$-space knot $K'$ with $\widehat{\HFK}(K')\cong\widehat{\HFK}(K)$.
\end{proposition}

\begin{proof}[Proof of Proposition~\ref{HFLcables}]
Applying~\ref{Hedd}, this result follows as in the proof of~\cite[Theorem 3.1]{binns2022cable}.
\end{proof}

We now give a characterizations of cables of $(m,mn)$ cables of almost $L$-space knots for sufficiently large $m$.

\begin{proposition}\label{almostLspacecables}
 Let $K$ be an almost $L$-space knot.  Suppose $L$ is a link such that $\widehat{\HFL}(L)\cong \widehat{\HFL}(K_{m,mn})$ with $\change{n}>2g(K)-1$. Then $L$ is the $(m,mn)$-cable of an almost $L$-space knot $K'$ such that $\widehat{\HFK}(K')\cong\widehat{\HFK}(K)$.
\end{proposition}

\begin{proof}
Suppose $K$ is as in the statement of the theorem. The same argument as used by the author and Dey in~\cite[Theorem 3.1]{binns2022cable} implies that $L$ is the $(m,mn)$-cable of some knot $K'$ such that $\change{\rank(\widehat{\HFK}(K'_n))=n+4}$ and $\Delta_{K'}(t)=\Delta_K(t)$.

We now show that $K'$ is either an $L$-space knot or an almost $L$-space knot. Let $\gamma$ be the immersed curve invariant of $K'$. Applying Proposition~\ref{Hanselmanrank}, we have that $\change{n}+4\geq \rank(\widehat{\HF}(S^3_{\change{n}}(K'))=|\change{n}-a|+b\geq \change{n}$ where $b$ is the number of vertical components in a singular pegboard diagram for $\gamma$ counted with multiplicity and $a$ is the slope of the unique segment which is not vertical. Suppose $\change{n}< a$. Then $a+b\leq 2\change{n}+4<2a+2$. $b-a$ is an even non-negative integer. It follows that $b=a+2$ or $b=a$. If $b=a$ then $K'$ is an $L$-space knot, as shown by an application of Proposition~\ref{Hanselmanrank}. If $b=a+2$ then $K'$ is an almost $L$-space knot as shown by an application of Proposition~\ref{Hanselmanrank}. If $\change{n}\geq a$ then we have that $b-a\leq4$. However, this yields only one new case, namely $b-a=4$. In this case\change{,} it can be observed from the immersed curve formula for $\widehat{\HFK}(K_{\change{n}})$ that $\rank(\widehat{\HFK}(K_{\change{n}}))\geq \change{n}+6$.

Suppose now that $K'$ is an $L$-space knot. Then $\change{n}=2g(K')-2$, as can be seen from Proposition~\ref{Hanselmanrank}. Since $\widehat{\HFL}(K'_{m,mn})\cong\widehat{\HFL}(K_{m,mn})$, $K_{m,mn}$ and $K'_{m,mn}$ have the same genus. Applying work of Gabai~\cite{gabai1986detecting}, Neumann-Rudolph~\cite{neumann1987unfoldings} and Rudolph~\cite{rudolph2002non}, it in turn follow\change{s} that $K$ and $K'$ have the same genus, contradicting $n>2g(K)-1$.

Thus $K'$ is an almost $L$-space knot. It remains to show that $\widehat{\HFK}(K')\cong\widehat{\HFK}(K)$. Observe that since $K$ and $K'$ share the same Alexander polynomial. Theorem~\ref{infinityclassification}  and Proposition~\ref{prop:delta0} imply that $\Delta_K(t)$ determines $\widehat{\HFK}(K)$ for almost $L$-space knots, completing the proof.
\end{proof}

\begin{corollary}\label{genus1almostLspacecables}
    Link Floer homology detects the $(m,mn)$-cables of $T(2,-3)$, the figure eight knot and the mirror of $5_2$ for $\change{n}>1$.
\end{corollary}

\begin{proof}
   Let $K$ be one of the knots in the statement of the proposition. Suppose $L$ is a link such that $\widehat{\HFL}(L)\cong\widehat{\HFL}(K_{m,mn})$ for some $n>1$. Proposition~\ref{almostLspacecables} implies that $L$ is as $(m,mn)$ cable of an almost $L$-space knot $K'$ with the same $\widehat{\HFK}(-)$ type as $K$. Proposition~\ref{genus1class} shows that there are only $3$ genus one almost $L$-space knots, namely those listed in the statement of the proposition. They are each distinguished by their Alexander polynomials, concluding the proof.
\end{proof}

\bibliographystyle{amsalpha}
\bibliography{bibliography}
\end{document}